\documentclass[
]
{amsart}
\usepackage{amsfonts,amsmath, amsmath,amsthm,amscd,
enumerate, 
xypic,amssymb, kpfonts, 
pdfsync,
a4}

\usepackage{setspace}

\newtheorem{thm}{Theorem}[section]

\newtheorem{prop}[thm]{Proposition}  
\newtheorem{cor}[thm]{Corollary}

\newtheorem{lem}[thm]{Lemma}
\theoremstyle{definition}
\newtheorem{dff}[thm]{Definition}

\newtheorem{remark}[thm]{Remark}

\newtheorem{notation}[thm]{Notation}

\theoremstyle{plain}
\renewenvironment{proof}{\par\noindent {\bf Proof. }}{\hfill$\Box$\bigskip}
\newcommand{\ennd}{\mathop{\mathrm{End}}}

\renewcommand{\phi}{\varphi}

\newcommand{\ve}{\varepsilon}

\newcommand{\id}{\mathop{\mathrm{id}}}

\renewcommand{\d}{{\mathop{\mathrm{d}}}_{\pm}}%

\newcommand{\jz}[2]
{\mbox{$J_{#1}^{#2}$}}
\newcommand{\intt}{\mathop{\mathrm{Int}}}
\medskip  

\author{M. G. Mahmoudi}
\title{
Orthogonal symmetries and Clifford algebras}

\begin{document}
 \maketitle

\begin{abstract} 
 Involutions of the Clifford algebra of a quadratic space induced by
 orth\-ogonal symmetries are investigated.\\ \\

\noindent
\emph{2000 Mathematics Subject Classification:}  11E04, 11E39, 11E88, 16W10 \\

\noindent
\emph{Key words:} 
Orthogonal symmetry, reflection, Clifford algebra, Involution,
quadratic form, even Clifford algebra, multiquaternion algebra, universal property of Clifford algebra,
Clifford map
\end{abstract}

\section{Introduction}
Clifford algebra is one of the important algebraic structures which
can be associated to a quadratic form. 
These algebras are among the most fascinating algebraic structures.
Not only they have many applications in algebra and other branches of
Mathematics, but also they have wide applications beyond Mathematics,
e.g., Physics,  Computer Science and
Engineering \cite{girard}, \cite{sabbata}, \cite{lounesto}, \cite{lounesto-ablamowicz}.   

A detailed historical account of Clifford algebras from their genesis
can be found in \cite{vdwaerden}. 
See also \cite{helmstetter-micali} and \cite{lounesto} for an interesting brief historical
account of Clifford algebras. 

Many familiar algebras can be regarded as special cases of Clifford
algebras. 
For example the algebra of Complex numbers is isomorphic to the
Clifford algebra of any one-dimensional negative definite quadratic form over $\mathbb{R}$.
The algebra of Hamilton quaternions is isomorphic to the Clifford algebras of a two-dimensional
negative definite quadratic form over $\mathbb{R}$. 
More generally, as shown by D. Lewis in \cite[Proposition
1]{lewis-clifford}, a multiquaternion algebra, i.e., an algebra like
$Q_1\otimes\cdots\otimes Q_n$ where $Q_i$ is a quaternion algebra over 
a field $K$, can be regarded as the Clifford algebra
of a suitable nondegenerate quadratic form $q$ over the base field
$K$.
In \cite{lewis-clifford}, such a form $q$ is also explicitly constructed.  
The Grassmann algebra (or the exterior algebra) may also be regarded as the
Clifford algebra of the null (totally isotropic) quadratic form.

The word \emph{`involution'} appears in different contexts in
Mathematics.
In group theory, an involution is an element of a group whose order is
two. 
In analysis, an involution of a (normed) Banach algebra $A$ is a map
$*$ from $A$ into itself such that for every $x,\ y\in A$: $x^{**}=x$, $(xy)^*=y^*x^*$ and
$\lVert x^*x\rVert=\lVert x\rVert^2$.
More generally in the theory of algebras, an involution of an algebra $A$ is a self
inverse map $\varphi$ from $A$ to itself such that for every $x$, $y\in A$, we
have $\varphi(x+y)=\varphi(x)+\varphi(y)$ and $\varphi(xy)=\varphi(y)\varphi(x)$, i.e., $\varphi$ is an
anti-automorphism of order $2$. 
Many algebras are naturally equipped with an involution, for example
the full matrix algebra $\mathbb{M}_n(K)$, consisting of all $n\times
n$ matrices over a field $K$ is equipped with the \emph{transposition} involution,
i.e., the involution which maps a matrix $A\in \mathbb{M}_n(K)$ to its
transpose, $A^t$.
If $q$ is a quadratic form defined on a vector space $V$, then the
endomorphism algebra $\mathrm{End}(V)$ is equipped with the adjoint
involution $\sigma_q:\mathrm{End}(V)\rightarrow \mathrm{End}(V)$
characterized by the property $b(x,fy)=b(\sigma_q(f)x,y)$ for every
$x,\ y\in V$ and for every $f\in\mathrm{End}(V)$, here $b$ is the
bilinear form associated to $q$.  
This involution, which determines $q$ up to similarity, reflects many properties of the quadratic form $q$. 
For example we have the following assertions: (1) $q$ is isotropic if and only if there exists $0\neq a\in
\mathrm{End}(V)$ such that $\sigma_q(a)a=0$. 
(2) $q$ is hyperbolic if and only if there exists an idempotent $e\in
\mathrm{End}(V)$ such that $\sigma_q(e)=1-e$,
see \cite{bfst} or \cite[Ch. II]{boi}. 

Clifford algebras have also many natural involutions.
In fact every orthogonal symmetry $s$ (i.e., a self-inverse isometry)
of a quadratic space $(V,q)$ can be extended to an
involution $J^s$ of both the Clifford algebra and the even Clifford algebra of $q$.
In particular the Clifford algebra of $q$ has two natural involutions which are
induced by the maps $\id:V\rightarrow V$ and $-\id:V\rightarrow V$.
The involutions $J^{\id}$ and $J^{-id}$ also reflect certain
properties of $q$; 
for instance their hyperbolicity is equivalent to the existence of
particular subforms of $q$.
(see \cite{mahmoudi}). 

Finite dimensional simple algebras with involution
form an important class of algebras with involution whose properties
are relatively well understood.
By a theorem due to Albert, a central simple $K$-algebra $A$ carries an
involution fixing $K$ if and only if the order of the class of $A$ in
its Brauer group of $K$ is at most two (see \cite[Ch. 8, 8.4]{scharlau}).      
As a consequence of a theorem due to Merkurjev, every central
simple algebra whose class in the Brauer group of $K$ is at most two is 
equivalent to a tensor product of quaternion algebras. 
An important class of algebras with involution is tensor products of
quaternion algebras with involution which are extensively studied in
the literature.  
There are some close analogies between the properties of tensor products
of quaternion algebras with involution on one hand and on the other hand the properties of multiples of Pfister forms.
See \cite{becher}, where it is for example proved that if $(A,\sigma)$
is a tensor product of quaternion algebras with involution such that
$A$ is Brauer equivalent to a quaternion algebra over field $K$ then for every field
extension $L/K$, $(A,\sigma)_L$ is either anisotropic or hyperbolic.   
This property is one of the characteristic properties of the multiples of Pfister forms. 

One can find tensor product of quaternion algebras with involution
which cannot occur as $(C(q),J^{\pm\id})$ for any form $q$ (see
comment after Proposition 3 in \cite{lewis-clifford}).
We however show that for every tensor product of quaternion
algebras with involution
$(A,\sigma)=(Q_1,\sigma_1)\otimes\cdots\otimes(Q_n,\sigma_n)$, there exists a
quadratic space $(V,q)$ of dimension $2n$ and an orthogonal symmetry
$\sigma:V\rightarrow V$  such that $(A,\sigma)$ is isomorphic to
$(C(q),J^\sigma)$, i.e., the
Clifford algebra of $q$ equipped with the involution induced by
$\sigma$ (see Theorem \ref{mul-qua-cliff-sym}).
We then provide more detailed statements; it turns out that when
$\sigma$ is of the first kind, the orthogonal symmetry $\sigma$ can be
chosen to be either $\pm\id$ or a reflection
(see Proposition \ref{tensor-odd} and Proposition \ref{tensor-even}).
For the case where $\sigma$ is of the second kind, a similar result is
proved (see Proposition \ref{tensor-second-kind}). 
Some results in this direction were already available in the
literature, see \cite[\S3]{kq} and \cite[Lemma 10.6]{shapiro}.
In order to prove these results, we need to provide involutorial
versions of some of the main structure theorems of Clifford algebras (i.e., the
results which link the Clifford algebra of an orthogonal sum of two
quadratic space to the Clifford algebras of summands, see
\cite[Ch. II]{chevalley} or \cite[Ch. V]{lam}). 
Section \ref{section-decomposition} is devoted to prove these
results. 
Some results in this direction were already been obtained in the
literature, see Proposition 2 of \cite{lewis-period}.

In the literature, the even Clifford algebra (i.e., the even sub-algebra of
the Clifford algebra) of a quadratic space $(V,q)$ is generally defined as a
sub-algebra of the Clifford algebra which is generated by products of
an even numbers of vectors in $V$.
As we will observe in section \ref{section-decomposition}, having
a definition of the even Clifford algebra, as an individual
mathematical object by means of a universal property, would be a handy
tool at our disposal.
Especially in proving isomorphisms of algebras with involution
involving even Clifford algebras, the universal property can slightly
shorten the proofs. 
This was our motivating reason to find a universal property of the
even Clifford algebra in section \ref{puacliff}.
We also hope that this approach will be useful from a pedagogical point
of view.  

In sections \ref{section-orsym} and \ref{section-type}, we make some general observations about
involutions of a Clifford algebra which are induced by an orthogonal
symmetry, for instance as an applications of results of section \ref{section-decomposition} the type of such involutions is determined.     
The type of natural involutions of the Clifford algebra $C(V,q)$ induced
by $\pm\id$ are known, see \cite{lewis-clifford} and \cite[pp. 116-118]{shapiro}. 

\section{Preliminaries}\label{s.preliminaires}

Let $K$ be a field of characteristic different from $2$.
A quadratic space $(V,q)$ over $K$ is a pair, consisting of a finite
dimensional vector space $V$ over $K$ and a quadratic form
$q:V\rightarrow K$, i.e., a map $q:V\rightarrow K$ which satisfies
$q(\lambda x)=\lambda^2 q(x)$ for all $\lambda\in K$ and $x\in V$
so that map $b_q: (x,y)\mapsto \frac{1}{2}(q(x+y)-q(x)-q(y))$ is a
bilinear map from $V\times V$ to $K$. 

A $K$-linear map from a quadratic space $(V,q)$ to a $K$-algebra $A$
with the unity element $1_A$ is called 
a \emph{Clifford map} if for every $x\in V$:
\begin{equation}\label{q(x).1}
\varphi(x)^2=q(x)\cdot 1_A.
\end{equation}

By replacing $x$ by $x+y$ in the relation  
(\ref{q(x).1}), we obtain 
\begin{equation}\label{q(x+y)}
\varphi(x)\varphi(y)+\varphi(y)\varphi(x)=2b_q(x,y)\cdot1_A,
\end{equation}
where $b_q(x,y)=\frac{1}{2}(q(x+y)-q(x)-q(y))$ is the bilinear form
associated to the quadratic form $q$. 
In particular the vectors $x$ and $y$ are orthogonal with respect to the bilinear
form $b_q$ if and only if their images
under the Clifford map $\varphi$ anticommute.\\

We recall that 
the \emph{Clifford algebra} of a quadratic space $(V,q)$ over a
field $K$ is an
algebra $C=C(V,q)$ over $K$ with the unity element 
$1_C$, endowed with a Clifford map  {$i_q$} from $(V,q)$ to 
$C$ such that:
\begin{itemize}
\item[( C1)] the $K$-algebra generated by $1_C$ and $i_q(V)$ is $C$;
\item[(C2)] for every Clifford map $\varphi:(V,q)\rightarrow A$, there exists an algebra homomorphism 
$\Phi: C\rightarrow A$ such that: $\varphi=\Phi\circ i_q$:
\end{itemize}   
\begin{equation}
\xymatrix@!{V\ar[rr]^{\varphi}\ar[dr]_{i_q} && A\\
&C\ar[ur]_{\Phi}&}
\end{equation}

Note that according to the properties $(C1)$ and $(C2)$, the map
$\Phi$ is unique and $i_q$ is also injective. 
The quotient of the tensor algebra $T(V)$ by the ideal generated by
all elements of the form $x\otimes x-q(x)\cdot 1$ where $x\in V$, satisfies this universal property.
The Clifford algebra of a quadratic space $(V,q)$ which is denoted 
by $C(V,q)$ therefore exists and it is uniquely determined (up to $K$-algebra isomorphism)
by the universal property $(C2)$.  
Fixing the injection map $i_q$, one may identify $V$ with a subspace
of $C(V,q)$.
The algebra $C(V,q)$ is also denoted by $C(q)$ if this leads to no
confusion. 
\\


Let $i_q$ be the injection map in the definition of the Clifford algebra
$C=C(V,q)$.
The map $-i_q$ defined by $(-i_q)(x)=-i_q(x)$
is also a Clifford map from $V$ to $C$, and thus there exists
an algebra homomorphism $\gamma:C\rightarrow C$
such that the following diagram commutes: 

\begin{equation} 
\begin{array}{lll}
\xymatrix@!{V\ar[rr]^{-i_q}\ar[dr]_{i_q} && C\\
&C\ar[ur]_{\gamma}&}& & 
\begin{array}{l}\\\\
\gamma(1_C)=1_C,\\
\gamma\circ i_q=-i_q,\\
\gamma(i_q(x))=-i_q(x).
\end{array}
\end{array}
\end{equation}

As $\gamma^2(i_q(x))=i_q(x)$, the map $\gamma^2$ coincides with the identity automorphism 
of $C$.\\ 

Let $C^+$ be the set of all elements $x\in C$ such that 
$\gamma(x)=x$. 
The set $C^+$ which is also denoted by $C_0(V,q)$ or $C_0(q)$, is a subalgebra of $C$ which is called the 
{\it even subalgebra} of $C$.
This subalgebra contains every element of $C(V,q)$ which is a product of an even number
of vectors of $V$, i.e., the products like $i_q(x_1)i_q(x_2)\cdots
i_q(x_{2p})$ where $x_1,\cdots,x_{2p}\in V$.
\\

Let $C^-$ be the set of all elements $x\in C$ such that
$\gamma(x)=-x$. 
The set $C^-$ is a subvector space $C(V,q)$ which is called the odd
part of the Clifford algebra $C(V,q)$. 
This subspace contains contains every products
of odd numbers of vectors of $V$.\\ 

Let $C^{op}$ be the opposite algebra of $C=C(V,q)$. 
The map
$i_q^*: V\rightarrow C^{op}$ defined by  
$i_q^*(x)=i_q(x)$
is also a Clifford map. 
There exists so a $K$-algebra
  homomorphism $J:C\rightarrow C^{op}$ such that
  $J\circ i_q=i_q^*$.
As
$i_q(V)=i_q^*(V)$,  
the map $J$ is a bijection.  
The map $J$ is the unique antiautomorphism of $C$
fixing $i_q(V)$ pointwise. 
The image  of
a product  $i_q(x_1)\cdots i_q(x_p)$ under the map $J$ is the product 
of the same terms in the inverse order.
The map $J$ is an important involution
of $C$ which is sometimes called the \emph{reversion} of $C(V,q)$
(cf. \cite[p. 107]{helmstetter-micali}). 
In general, an involution of a ring $R$ is an antiautomorphism of $R$
of order $2$. 
If $S$ is the subring of $R$ and if $S$ is a subset of
$\sigma$-invariant elements of $R$ then $\sigma$ is
called an $S$-involution. 
In particular, if $A$ is a $K$-algebra, then a $K$-involution of $A$ is
an involution of $A$ fixing $K$ elementwise.\\

\section{A universal property of the even Clifford algebra}
\label{puacliff}  
The aim of this section is to prove that the even Clifford algebra of
a 
quadratic space satisfies a 
certain  universal property. 
\begin{dff}\label{application-paire}
Let $(V,q)$ be a quadratic space over a field $K$ and let $A$ be an
algebra  
over $K$ with the unity element $1_A$. 
An \emph{even Clifford map} is a bilinear map $\psi$ from
$V\times V$ to $A$ such that for all $x,\ y,\ z\in V$   
\begin{itemize}
\item[1)] $\psi(x,y)\psi(y,z)=q(y)\cdot \psi(x,z)$
\item[2)] $\psi(x,x)=q(x)\cdot 1_A$
\end{itemize}
\end{dff}

\begin{remark}
The conditions (1) and (2) given in the previous definition are 
respectively equivalent to the following conditions: 
\begin{itemize}
\item[1')] $\psi(x,y)\psi(x,z)=2b_q(x,y)\cdot \psi(x,z)-q(x)\cdot \psi(y,z)$
\item[2')] $\psi(x,y)+\psi(y,x)=2b_q(x,y)\cdot 1_A$,
\end{itemize}
where $b_q$ is the associated quadratic form to $q$. 
In particular if $x,y$ are orthogonal with respect to the bilinear
form $b_q$ then  
 $\psi(x,y)\psi(x,z)=-q(x)\cdot\psi(y,z)$.
\end{remark}

\begin{thm}\label{clifford-paire}
  Let $C_0=C_0(V,q)$ be the even Clifford algebra of the nondegenerate quadratic space 
$(V,q)$ over a field $K$ with the unity element 
$1_{C_0}$. 
There exists an even Clifford map $j$ from
$V\times V$ to $C_0$ which satisfies the following conditions:
\begin{itemize}
\item[a)] as a $K$-algebra $C_0$ is generated by $1_{C_0}$ and \{$j(x,y),\ x,\ y\in V$\}.
\item[b)] for every even Clifford map
$\psi:V\times V\rightarrow A$, there exists a unique algebra homomorphism
$\Psi:C_0\rightarrow A$ such that $\psi=\Psi\circ j$, i.e., the
following diagram commutes:
\end{itemize} 
\begin{equation}\label{diag-universal-property}
\xymatrix@!{V\times V\ar[rr]^{\psi}\ar[dr]_{j} && A\\
  &C_0\ar[ur]_{\Psi}&}
\end{equation}
\end{thm}

\begin{proof}
Let $C=C(V,q)$ be the Clifford algebra of $(V,q)$ and let
$i_q:V\rightarrow C$ be the canonical injection of
$C(V,q)$.  
We claim that the map $j:V\times V\rightarrow C_0$
defined by $j(x,y)=i_q(x)i_q(y)$ 
is an even Clifford map which satisfies the universal property  
described in the statement of the theorem. 

The fact that $i_q$ is a Clifford map readily implies that $j$ is an
even Clifford map. 
Let $\psi:V\times V\rightarrow A$ be an arbitrary even Clifford map
satisfying: 

$$\left \lbrace\begin{array}{l}
\psi(x,y)\psi(y,z)=q(y)\cdot\psi(x,z)\\
\psi(x,x)=q(x)\cdot 1_A
\end{array}\right.$$ \\

We have to find a $K$-algebra homomorphism $\Psi:C_0\rightarrow A$ such that 
the diagram (\ref{diag-universal-property}) commutes. 

As $(V,q)$ is nondegenerate, there exists a vector $v\in V$ such that
$q(v)\neq 0$. 
Take 
$d:=b_q(v,v)=q(v)\neq 0$ and let $V_0={v}^{\perp}=\{x\in V: b_q(x,v)=0\}$. 
The Witt decomposition 
$q\simeq \langle d\rangle\perp q_0$ where $q_0$ is a subform of $q$
with $\dim q_0=\dim q-1$ induces a
decomposition of vector spaces $V=(K\cdot v)\perp V_0$ where $V_0$ is a
subspace of $V$ of codimension $1$. 

Let
$f:V_0\rightarrow A$ and $g:V_0\rightarrow C_0$ the
maps defined by $f(w)=\psi(v,w)$ and
$g(w)=i_q(v)i_q(w)$ for every $w\in V_0$. 
We have
$f(w)^2=\psi(v,w)^2=-q(v)q(w)\cdot 1_{A}=-dq_0(w)\cdot 1_{A}$
and $g(w)^2=(i_q(v)i_q(w))^2=-dq_0(v)\cdot 1_C=-dq_0(v)\cdot 1_{C_0}$. 
Thus $f:(V_0,-dq_0)\rightarrow A$ and $g:(V_0,-dq_0)\rightarrow C_0$
are two Clifford maps.
It follows that there exist unique homomorphisms 
$F:C(V_0,-dq_0)\rightarrow A$ and
$G:C(V_0,-dq_0)\rightarrow C_0$ such that the following diagram
commutes:
\begin{equation}
\xymatrix{&C_0&\\
V_0\ar[ur]^{g}\ar[dr]_{f}\ar[rr]^{\hspace{-.5cm}i_{-dq_0}}&&C(V_0,-dq_0)
\ar[ul]_{G}\ar[dl]^{F}\\
&A&
}
\end{equation}
For all $w,\ w'\in V_0$ we have 
$g(w)g(w')=-di_q(w)i_q(w')$.
Therefore for all $w,\ w'\in V_0$, the image of $G$ contains the
elements $i_q(w)i_q(w')$. 
But $C_0$ is generated by the elements
$i_q(x)i_q(y)$, $x,\ y\in V$ and $1_{C_0}$ and the element  
$i_q(x)i_q(y)$ can be written as a linear combination of the 
elements $i_q(w)i_q(w')$,  $i_q(v)i_q(w)$, $1_C$ for suitable
$w,\ w'\in V_0$. 
These elements are in 
the image of $G$. 
Therefore the homomorphism $G$ is surjective.
For dimension reasons it is also injective, it is therefore an isomorphism. 
Consider the map 
$\Psi:C_0\rightarrow A$ defined by  
$$\Psi=F\circ G^{-1}.$$
We claim that the diagram (\ref{diag-universal-property}) commutes.
Let $(v_1,v_2)\in V\times V$.
We can write 
$(v_1,v_2)=(\lambda_1v+w_1,\lambda_2v+w_2)$ where 
$\lambda_1,\ \lambda_2\in K$ and $w_1,\ w_2\in V_0$. 
We have:\\

$\begin{array}{ll}
\Psi({j}(v_1,v_2))&=F\circ
G^{-1}({j}(v_1,v_2))\\
&=F\circ G^{-1}(i_q(v_1)i_q(v_2))\\
&=F\circ G^{-1}(i_q(\lambda_1v+w_1)i_q(\lambda_2v+w_2))\\
&=F\circ G^{-1}
(\lambda_1\lambda_2q(v)+\lambda_2i_q(w_1)i_q(v)\\
&\phantom{=F\circ G^{-1}
(} +\lambda_1i_q(v)i_q(w_2)+i_q(w_1)i_q(w_2))\\
&=F\circ G^{-1}
(\lambda_1\lambda_2q(v)-\lambda_2i_q(v)i_q(w_1)\\
&\phantom{=F\circ G^{-1}
(} +\lambda_1i_q(v)i_q(w_2)+i_q(w_1)i_q(w_2))\\
&=F\circ G^{-1}
(\lambda_1\lambda_2q(v)-\lambda_2g(w_1)\\
&\phantom{=F\circ G^{-1}(} 
+\lambda_1g(w_2)-d^{-1}g(w_1)g(w_2))\\
&=F
(\lambda_1\lambda_2q(v)-\lambda_2i_{-dq_0}(w_1)\\
&\phantom{=F\circ G^{-1}(} 
+\lambda_1i_{-dq_0}(w_2)-d^{-1}i_{-dq_0}(w_1)i_{-dq_0}(w_2))\\
&=\lambda_1\lambda_2q(v)-\lambda_2\psi(v,w_1)\\
&\phantom{=F\circ G^{-1}(} 
+\lambda_1\psi(v,w_2)-d^{-1}\psi(v,w_1)\psi(v,w_2)\\
&=\lambda_1\lambda_2q(v)+\lambda_2\psi(w_1,v)
+\lambda_1\psi(v,w_2)+\psi(w_1,w_2)\\
&=\psi(\lambda_1v+w_1,\lambda_2v+w_2)\\
&=\psi(v_1,v_2).
\end{array}$
 
\end{proof}

\begin{remark}  
As an immediate consequence, we obtain that the algebra $C_0=C_0(V,q)$
is uniquely determined by the universal property described in the
statement of the previous theorem.
In fact if $(C_0,j)$ and $(C_0',j')$ are two pairs satisfying the
universal property proved in the previous theorem, then there exist two
$K$-algebra homomorphisms $J$ and $J'$ such that the following diagram commutes:

\begin{equation}
\xymatrix@!{V\times V\ar[rr]^{j'}\ar[dr]_{j} && 
C'_0 \ar@<2pt>[dl]^{J'}\\
&C_0\ar@<2pt>[ur]^{J}&}
\end{equation}
The relations $J\circ j=j'$ and
$J'\circ j'=j$ imply that: 
$J'\circ J\circ j=J'\circ j'=j$
and $J\circ J'\circ j'=J\circ j=j'$.
As $C_0$ is generated, as a $K$-algebra, by
$1_{C_0}$ and $j(V\times V)$ and  $C'_0$ is generated, as a
$K$-algebra, 
by $1_{C'_0}$ and $j'(V\times V)$, 
we obtain
$J'\circ J=\id_{C_0}$ and 
 $J\circ J'=\id_{C'_0}$. 
\end{remark}

\section{Involutions induced by
an orthogonal symmetry}\label{section-orsym}

\begin{dff}
Let $V$ be a finite dimensional vector space over a field $K$ endowed 
with a nondegenerate (symmetric or anti-symmetric) bilinear form $b$. 
An \emph{isometry} of $(V,q)$, i.e., an element $\sigma\in \ennd(V)$ such that $b(\sigma
x,\sigma y)=b(x,y)$ for all $x,\ y\in V$ is said to be a
\emph{orthogonal symmetry} if $\sigma^2=\id$.
In the literature, such
maps are sometimes also called  ``orthogonal involutions'' (cf. 
\cite[Ch.III, \S5]{deheuvels}). 
We have, however, preferred to use the former term in order to
avoid any possible confusion with already well-established notions of orthogonal, symplectic and
unitary involutions (see \cite{boi}).  
A \emph{reflection} $\tau$ of $(V,q)$ is an orthogonal symmetry whose
invariant subspace $V^+=\{x\in V:\ \tau(x)=x\}$ is a hyperplane of 
$V$. 
\end{dff}

We recall the following result from linear algebra:

\begin{prop}
\label{symetrie-car}
Let $(V,b)$ be a nondegenerate symmetric or anti-symmetric bilinear space over a field $K$.
Let $\sigma\in \ennd(V)$ be an element with $\sigma^2=\id$. 
Let $V^+$ and $V^-$ be respectively the eigenspaces of $\sigma$ associated to the
eigenvalues $+1$ and $-1$. 
Then: 
\begin{itemize}
\item[{\rm 1)}] the space $V$ is the direct sum of $V^+$ and $V^-$. 
\item[{\rm 2)}] the map $\sigma$ is an orthogonal symmetry if and only
  if ${V^+}$ and ${V^-}$ are orthogonal with respect to the form $b$.
\end{itemize}
See \rm{\cite[Ch. III, \S5]{deheuvels} or \cite[Ch. I, \S3]{dieudonne}}.
\end{prop}

\begin{prop}\label{cliff-sym}
Let $(V,q)$ be a quadratic space over a field $K$ and let $s\in \ennd{V}$.
In order that there exists a $K$-involution $\sigma$ of $C(V,q)$ such
that for every $x\in V$, 
$\sigma(i_q(x))=i_q(s(x))$, it is necessary and
sufficient that
$s$ be an orthogonal symmetry of $(V,q)$.
\end{prop}

\begin{proof}
Suppose that there exists a $K$-involution $\sigma$ of $C(V,q)$ such
that for every $x\in V$, 
$\sigma(i_q(x))=i_q(s(x))$.  
As
$q(x)=\sigma(q(x))=\sigma(i_q(x)^2)=i_q(s(x))^2=q(s(x))$, the map $s$ is an isometry of
$(V,q)$.
We also have 
$i_q(s^2(x))=\sigma(i_q(s(x)))=\sigma(\sigma(i_q(x)))=i_q(x)$. 
The injectivity of $i_q$ implies that  
$s^2=\id$. 
Therefore $s$ is an orthogonal symmetry.

Conversely suppose that $s$ is an orthogonal symmetry. 
Consider the map $\varphi:V\rightarrow C(V,q)$ defined by 
$\varphi(x)=i_q(s(x))$ for every $x\in V$. 
As $\varphi(x)^2=i_q(s(x))^2=q(s(x))=q(x)$, $\varphi$ is
a Clifford map. 
According to the universal property of the Clifford algebra,
there exists a unique homomorphism ${\Phi}:C(V,q)\rightarrow
C(V,q)$ such that ${\Phi}(i_q(x))=\varphi(x)$ for every $x\in V$.
Note that the image of $\Phi$ contains $i_q(V)$. 
As $C(V,q)$ is generated by $i_q(V)$, the homomorphism $\Phi$ is
surjective, hence it is an isomorphism. 

Let $j$ be the reversion involution of $C(V,q)$ as defined in 
\S\ref{s.preliminaires}. 
Consider the map $\sigma:C(V,q)\rightarrow C(V,q)$ defined by $\sigma=j\circ{\Phi}$. 
The map 
$\sigma$ is an anti-automorphism because $j$ is an anti-automorphism
and ${\Phi}$ is an isomorphism. 
We have $\sigma^2=\id$, in fact let $x\in V$, we obtain
$\sigma^2(i_q(x))=j\circ{\Phi}\circ j\circ{\Phi}(i_q(x))=
j\circ{\Phi}\circ j(\varphi(x))=
j\circ{\Phi}(\varphi(x))=
j\circ{\Phi}(i_q(s(x)))=
j(\varphi(s(x)))=
\varphi(s(x))=
i_q(s^2(x))=
i_q(x)$.
Thus we have shown that $\sigma$ is a $K$-involution of $C(V,q)$. 
We also have
$\sigma(i_q(x))=j\circ\Phi(i_q(x))=j(\phi(x))=i_q(s(x))$. 
This completes the proof.
\end{proof}

\begin{notation}
Let $s$ be an orthogonal symmetry of a quadratic space $(V,q)$.
From now on, the unique involution of the Clifford algebra $C(V,q)$
which maps $i_q(x)$ to $i_q(s(x))$, for every $x\in V$, is denoted by
$J_q^s$.
Compare with \cite[Ch. 3, 3.15; Ch. 4, 4.3]{shapiro}
\end{notation}

\begin{remark}\label{remark-deg-2}
We recall that the \emph{canonical involution} $\gamma$ of a quaternion algebra
$Q=(a,b)_K$ is the involution $\gamma:Q\rightarrow Q$, defined by $\gamma(a+bi+cj+dk)=a-bi-cj-dk$ where
$a,\ b,\ c,\ d\in K$ and $i,\ j, \ k$ are the generators of $Q$ with
$i^2=a$, $j^2=-b$ and $ij=-ji=k$.   
When $q$ is a nondegenerate quadratic form of dimension $2$, $C(V,q)$ is a
quaternion algebra.
In this case, the canonical involution of $C(V,q)$ coincides with the involution $J_q^{-\id}$. 
\end{remark}

\begin{cor}\label{cliff-paire-sym}
Let $(V,q)$ be a quadratic space over a field $K$, let $s$ be an
orthogonal symmetry of $V$ and let $j_q:V\times V\rightarrow
C_0(V,q)$ be the even Clifford map which satisfies the universal
property described in the statement of theorem \ref{clifford-paire}.   
Then there exists a unique $K$-involution of $C_0(V,q)$, again denoted
by $J^s_q$, such that $J^s_q({j_q}(x,y))={j_q}(s(y),s(x))$ for all $x,\ y\in V$.
\end{cor}

\begin{proof}
The involution $J^s_q$ of $C_0(V,q)$ is exactly the restriction of the involution $J_q^s$ of $C(V,q)$. 
One can also give a direct proof using the universal property of
$C_0(V,q)$.~ 
\end{proof}

\section{Decomposition of involutions induced by an
orthogonal symmetry}\label{section-decomposition}
We recall the following well known result:
\begin{lem}\label{lemme-jz}
Let $s$ be an orthogonal symmetry of a quadratic space $(V,q)$.
Let $\{e_1,\cdots,e_r\}$ be an orthogonal basis of ${V^+}$ and let
$\{f_1,\cdots,f_s\}$ be an orthogonal basis of ${V^-}$ (cf. \ref{symetrie-car}).
Consider the element $z=i_q(e_1)\cdots i_q(e_r)\cdot i_q(f_1)\cdots i_q(f_s)\in C(V,q)$. 
Then:
\begin{itemize}
\item[\rm{1)}] we have 
$J_q^s(z)=(-1)^s(-1)^{\frac{n(n-1)}{2}}z$, where $n=r+s=\dim V$, in
particular if $s$ is a reflection we have $J_q^{s}(z)=(-1)^{\frac{n(n+1)}{2}}z$.
\item[\rm{2)}] we have $z^2=\d q$.
\item[\rm{3)}] when $n$ is even, $z$ anti-commutes with every  
element $i_q(x)$, where $x\in V$.
\item[\rm{4)}] when $n$ is odd, $z$ commutes with every 
element $i_q(x)$, where $x\in V$. 
\end{itemize}
\end{lem}

\begin{proof}
The set  $\{e_1,\cdots,e_r, f_1,\cdots,f_s\}$  is an
orthogonal basis of $V$ because according to \ref{symetrie-car}, ${V^+}$ and ${V^-}$ are 
orthogonal. 
We have:\\

$\begin{array}{ll}
J_q^s(z)&=((-1)^si_q(f_s)\cdots i_q(f_1))
 (i_q(e_r)\cdots i_q(e_1))\\
&=(-1)^s(-1)^{\frac{n(n-1)}{2}}
i_q(e_1)\cdots i_q(e_r)i_q(f_1)\cdots i_q(f_s)\\
&=(-1)^s(-1)^{\frac{n(n-1)}{2}}z.
\end{array}$\\

The assertions (2), (3) and (4) are 
simple  well known calculations, see for instance \cite[Ch.V, \S2 ]{lam}.
\end{proof}

\begin{prop}\label{iso-cliff-paire-inv}
Let $(V_1,q_1)$ be a quadratic space of even dimension over a field $K$, let
$(V,q)$ be an arbitrary quadratic space, let $\sigma_1$ be an
orthogonal symmetry
of $(V_1,q_1)$ and let $\sigma$ be an orthogonal symmetry of $(V,q)$. 
Let $V_1$ be the vector space of $\sigma_1$-anti-symmetric elements of
$V_1$ and suppose that $\dim {V_1}^-=s$ (cf. \ref{symetrie-car}). 
Then:
\begin{itemize}
\item[\rm{(a)}] If $\dim V_1\equiv 1 \mod 4$ then we have:
\begin{equation*}(C_0(V_1\perp V,q_1\perp q),\jz{q_1\perp q}
{\sigma_1\oplus\sigma})\simeq 
(C_0(V_1,q_1)\otimes C(V,-\d q_1{\cdot}q), \jz{q_1}{\sigma_1}\otimes 
J_{-\d q_1{\cdot}q}^{(-1)^{s+1}\sigma}),\end{equation*}
in particular if $\sigma_1=\tau$ is a reflection then we have:
\begin{equation*}(C_0(V_1\perp V,q_1\perp q),\jz{q_1\perp q}
{\tau\oplus\sigma})\simeq 
(C_0(V_1,q_1)\otimes C(V,-\d q_1{\cdot}q), \jz{q_1}{\tau}\otimes 
J_{-\d q_1{\cdot}q}^{\sigma}).\end{equation*}
\item[\rm{(b)}] If  $\dim V_1\equiv 3 \mod 4$ then we have:
\begin{equation*}(C_0(V_1\perp V,q_1\perp q),\jz{q_1\perp q}
{\sigma_1\oplus\sigma})\simeq 
(C_0(V_1,q_1)\otimes C(V,-\d q_1{\cdot}q), \jz{q_1}{\sigma_1}\otimes 
J_{-\d q_1{\cdot}q}^{(-1)^{s}\sigma}),\end{equation*}
in particular if $\sigma_1=\tau$ is a reflection then we have:
\begin{equation*}(C_0(V_1\perp V,q_1\perp q),\jz{q_1\perp q}
{\tau\oplus\sigma})\simeq 
(C_0(V_1,q_1)\otimes C(V,-\d q_1{\cdot}q), \jz{q_1}{\tau}\otimes 
J_{-\d q_1{\cdot}q}^{-\sigma}).\end{equation*}
\end{itemize}
\end{prop}

\begin{proof} 
Let ${j_{q_1\perp q}}:(V_1\perp V)\times(V_1\perp V)\rightarrow 
C_0(V_1\perp V,q_1\perp q)$ be a map which satisfies the universal
property discussed in \ref{clifford-paire}.
Let $z\in C(V_1,q_1)$ be the element defined in  
\ref{lemme-jz}.
Let ${\eta}:(V_1\perp V)\times (V_1\perp V)\rightarrow 
C_0(V_1,q_1)\otimes C(V,-\d q_1{\cdot}q)$ be the map defined by
$${\eta}(x_1\oplus x, y_1\oplus y)=(i_{q_1}(x_1)\otimes 1+
z^{-1}\otimes i_{-\d q_1{\cdot}q}(x))\times
(i_{q_1}(y_1)\otimes 1-z^{-1}\otimes i_{-\d q_1{\cdot}q}(y)).$$
It is easy to verify that ${\eta}$ is an even Clifford map. 
There exists so a homomorphism
$${H}:C_0(V_1\perp V,q_1\perp q)\rightarrow 
C_0(V_1,q_1)\otimes C(V,-\d q_1{\cdot}q)$$ such that the following diagram 
commutes: 
\begin{equation}
\xymatrix{(V_1\perp V)\times (V_1\perp V)\ar[r]^{{j_{q_1\perp q}}}
\ar[dr]_{{\eta}} &C_0(V_1\perp V,q_1\perp q)\ar[d]^{{H}}\\
&C_0(V_1,q_1)\otimes C(V,-\d q_1{\cdot}q)}
\end{equation}

We now prove that ${H}$ is an isomorphism. 
As 
$$\dim_K C_0(V_1\perp V,q_1\perp q)=\dim_K C_0(V_1,q_1)\otimes C(V,-\d q_1{\cdot}q),$$
the homomorphism ${H}$ is surjective if and
only if it is injective. 
We show that ${H}$ is surjective. 
The algebra $C_0(V_1,q_1)\otimes C(V,-\d q_1{\cdot}q)$
is generated by all elements $i_{q_1}(x_1)i_{q_1}(y_1)\otimes 1$ and 
$1\otimes i_{-\d q_1{\cdot}q}(y)$ where $x_1,\ y_1\in V_1$ and
$y\in V$. 
It is enough to show that these elements lie
in the image of ${H}$. 

We have: $${\eta}(x_1\oplus 0,y_1\oplus 0)=i_{q_1}(x_1)i_{q_1}(y_1)\otimes 1,$$
thus the element $i_{q_1}(x_1)i_{q_1}(y_1)\otimes 1$ lie in the image
of ${H}$. 

On the other hand, 
${\eta}(x_1\oplus 0,0\oplus y)=i_q(x_1)z^{-1}\otimes i_{-\d
  q_1{\cdot}q}(y).$
By definition, the element $z$ is of the form 
$i_{q_1}(e_1)i_{q_1}(e_1)\cdots i_{q_1}(e_{2m+1})$ 
where $\{e_1,$ $e_2,$ $\cdots,$ $e_{2m+1}\}$
is a basis of $V_1$. 
Therefore the element 
$$ i_q(e_{2m})^{-1}\cdots i_q(e_2)^{-1}i_q(e_1)^{-1}
\otimes i_{-\d q_1{\cdot}q}(y)=
i_q(e_{2m+1})z^{-1}\otimes i_{-\d q_1{\cdot}q}(y)$$ lies in the image
of ${H}$. 
As $i_q(e_1)i_q(e_2)\cdots i_q(e_n)\otimes 1$ also lies
to the image of ${H}$, we deduce that $1\otimes i_{-\d q_1{\cdot}q}(y)$ 
lies in the image of ${H}$.

Suppose that  $\dim V_0\equiv 1 \mod 4$, to prove 
(a), it is enough to verify that the isomorphism $$C_0(V_1\perp V,q_1\perp q)\simeq 
(C(V_0,q_0)\otimes C(V,-\d q_1{\cdot}q)$$ is compatible with the involutions 
 $\jz{q_1\perp q}{\sigma_1\oplus\sigma}$ and  $\jz{q_1}{\sigma_1}\otimes 
J_{-\d q_1{\cdot}q}^{(-1)^{s+1}\sigma}$. 
For all $x_1,$ $y_1$ $\in$ $V_1$ and $x,$ $y$ $\in$ $V$ we have:
\begin{eqnarray*}
{H}\circ \jz{q_1\perp q}{\sigma_1\oplus\sigma}({j_{q_1\perp q}}
(x_1\oplus x,y_1\oplus y)) &=&{H}({j_{q_1\perp q}}(\sigma_1(y_1)\oplus\sigma(y),
\sigma_1(x_1)\oplus\sigma(x)))\\
                          &=&{\eta}(\sigma_1(y_1)\oplus\sigma(y),\sigma_1(x_1)\oplus\sigma(x)))\\
                          &=&(i_{q_1}(\sigma(y_1))\otimes 1+z^{-1}\otimes i_{-\d
  q_1{\cdot}q}(\sigma(y)))\\
                          &&\times
(i_{q_1}(\sigma(x_1))\otimes 1-z^{-1}\otimes i_{-\d q_1{\cdot}q}(\sigma(x)))
\end{eqnarray*}

On the other hand, if $A=\jz{q_1}{\sigma_1}\otimes J_{-\d q_1{\cdot}q}^{(-1)^{s+1}\sigma}\circ{H}
({j_{q_1\perp q}}(x_1\oplus x,y_1\oplus y))$, using \ref{lemme-jz} we obtain\\

\begin{eqnarray*}
\lefteqn{\jz{q_1}{\sigma_1}\otimes J_{-\d q_1{\cdot}q}^{(-1)^{s+1}\sigma}\circ{H}
({j_{q_1\perp q}}(x_1\oplus x,y_1\oplus y))}\\ &=&J_{q_1}^{\sigma_1}\otimes J_{-\d q_1{\cdot}q}^{(-1)^{s+1}\sigma}
\left((i_{q_1}(x_1)\otimes 1\right.\\
&&+
z^{-1}\otimes i_{-\d q_1{\cdot}q}(x))\\
&&\times
\left.(i_{q_1}(y_1)\otimes 1-z^{-1}\otimes i_{-\d
  q_1{\cdot}q}(y))\right)\\
&=&(i_{q_1}(\sigma_1(y_1))\otimes 1\\
&&-(-1)^s z^{-1}\otimes i_{-\d
  q_1{\cdot}q}((-1)^{s+1}\sigma(y))) \\ 
&&\times
(i_{q_1}(\sigma_1(x_1))\otimes 1\\
&&+(-1)^s z^{-1}\otimes i_{-\d
  q_1{\cdot}q}((-1)^{s+1}\sigma(x)))\\
&=&(i_{q_1}(\sigma(y_1))\otimes 1+z^{-1}\otimes i_{-\d
  q_1{\cdot}q}(\sigma(y)))\\
&&\times
(i_{q_1}(\sigma(x_1))\otimes 1-z^{-1}\otimes i_{-\d q_1{\cdot}q}(\sigma(x))).
\end{eqnarray*}

We thus have:
$$(\jz{q_1}{\sigma_1}\otimes J_{-\d q_1{\cdot}q}^{(-1)^{s+1}\sigma})\circ{H}
= {H}\circ \jz{q_1\perp q}{\sigma_1\oplus\sigma}$$
It follows that ${H}$ is an automorphism of algebras with involution.

The proof of (b) is similar.
\end{proof}

\begin{remark}
The isomorphism 
\begin{equation}\label{clifford-paire-sans-involution}
C_0(V_1\perp V,q_1\perp q)\simeq 
(C_0(V_1,q_1)\otimes C(V,-\d q_1{\cdot}q)
\end{equation} is known. (cf. \cite[Ch.V, \S 2]{lam}). 
Proposition \ref{iso-cliff-paire-inv} is so an involutorial version of
(\ref{clifford-paire-sans-involution}).
The classical proof 
of (\ref{clifford-paire-sans-involution}) given in \cite{lam} is
different from what we have provided and uses the properties
of graded simple algebras and graded tensor product. 
\end{remark}

\begin{cor}\label{cor-cliff-paire-inv}
Keeping the same hypotheses as in \ref{iso-cliff-paire-inv}:
\begin{itemize}
\item[\rm{(a)}] If $\dim V_1\equiv 1 \mod 4$ then we have:
\begin{equation*}
(C_0(V_1,q_1)\otimes C(V,q), \jz{q_1}{\sigma_1}\otimes 
J_{q}^{\sigma})\simeq 
(C_0(V_1\perp V,q_1\perp-\d q_1{\cdot}q),\jz{q_1\perp-\d q_1{\cdot}q}
{\sigma_1\oplus(-1)^{s+1}\sigma}),\end{equation*}
in particular if $\sigma_1=\tau$ is a reflection then we have: 
\begin{equation*}
(C_0(V_1,q_1)\otimes C(V,q), \jz{q_1}{\tau}\otimes 
J_{q}^{\sigma})\simeq 
(C_0(V_1\perp V,q_1\perp-\d q_1{\cdot}q),\jz{q_1\perp-\d q_1{\cdot}q}
{\tau\oplus\sigma}).\end{equation*}
\item[\rm{(b)}] If  $\dim V_1\equiv 3 \mod 4$ then we have:
\begin{equation*}
(C_0(V_1,q_1)\otimes C(V,q), \jz{q_1}{\sigma_1}\otimes 
J_{q}^{\sigma})\simeq 
(C_0(V_1\perp V,q_1\perp-\d q_1{\cdot}q),\jz{q_1\perp-\d q_1{\cdot}q}
{\sigma_1\oplus(-1)^{s}\sigma}),\end{equation*}
in particular if $\sigma_1$ is a reflection then we have:
\begin{equation*}
(C_0(V_1,q_1)\otimes C(V,q), \jz{q_1}{\tau}\otimes 
J_{q}^{\sigma})\simeq 
(C_0(V_1\perp V,q_1\perp-\d q_1{\cdot}q),\jz{q_1\perp-\d q_1{\cdot}q}
{\tau\oplus-\sigma}).\end{equation*}
\end{itemize} 
\end{cor}

\begin{cor}\label{C0C-id}
For every quadratic form $q$ over a field $K$ and for every 
$d\in K^*$, there exists an isomorphism of algebras with involution
\begin{equation}(C_0(\langle -d\rangle\perp q),\jz{}{\pm\id})\simeq
(C(d{\cdot}q),J^{-\id}).\end{equation}
 In particular, 
\begin{equation}(C_0(\langle -1\rangle\perp q),\jz{}{\pm\id})\simeq
(C(q),J^{-\id}).\end{equation}
\end{cor}

\begin{cor}\label{C0Cid}
Keeping the notation of \ref{C0C-id}, let $V_1$ and $V$ be respectively the underlying
vector spaces of $\langle-d\rangle$ and $q$. 
Consider the reflection $\tau=-\id_{V_1}\oplus\id_{V}$. 
Then we have an isomorphism of algebras with involution  
\begin{equation}(C_0(\langle -d\rangle\perp q),\jz{}{\tau})\simeq
(C(d{\cdot}q),J^{\id}).\end{equation}
 In particular, 
\begin{equation}(C_0(\langle -1\rangle\perp q),\jz{}{\tau})\simeq
(C(q),J^{\id}).\end{equation}
\end{cor}

\begin{cor}\label{cor-cliff-iso}
Let $(V,q)$ be a quadratic space over a field $K$, let 
$\sigma$ be an orthogonal symmetry of $V$ and let $v\in V$ be an
anisotropic vector such that $\sigma(v)=\ve v$ where $\ve=1\text{ ou
}-1$. 
Let 
$V'=\{x\in V : b_q(x,v)=0\}$ the orthogonal complement of the subspace
generated by $v$ and let 
$d=q(v)$. 
Then $V'$ is stable under $\sigma$ and we have an isomorphism
of algebras with involution
\begin{equation}(C_0(V,q),\jz{q}{\sigma})\simeq (C(V',-d{\cdot}q'),J_{-d{\cdot}q'}^{\sigma'}),\end{equation} where
$\sigma'=-\ve\sigma|_{V'}$ and $q'=q|_{V'}$.
\end{cor}

\begin{proof}
It is enough to note that, $V'$ is stable under $\sigma$. 
We have the decomposition $(V,\sigma)=(v\cdot K\oplus V',\pm\id\oplus\sigma|_{V'})$
and we can use Proposition \ref{iso-cliff-paire-inv}.
\end{proof}

\begin{cor}
Let $(V,q)$ be a nondegenerate quadratic space over a field $K$ and let $a\in K^*$.
Then for every orthogonal symmetry $\sigma$ of $(V,q)$ 
we have an isomorphism of algebras with involution
\begin{equation}(C_0(V,q),\jz{q}{\sigma})\simeq (C_0(V,a{\cdot}q),\jz{a{\cdot}q}{\sigma}).\end{equation}
\end{cor}

\begin{proof}
Let ${V^+}$ and ${V^-}$ be respectively the subspaces of the symmetric
and anti-symmetric elements of $V$ (cf. \ref{symetrie-car}). 
As $q$ is nondegenerate, it follows that either ${V^+}$ or ${V^-}$ contain an anisotropic vector
$v$.
Thus there exists $v\in V$ such that 
$\sigma(v)=\ve v$ where $\ve=\pm1$ and $b=q(v)\in K^*$.  
Let $V'$ be the subspace of $V$ consisting of all
elements, orthogonal to $v$ and let $q'=q|_{V'}$. 
Using Proposition \ref{cor-cliff-iso} we obtain 
$$(C_0(V,a{\cdot}q),\jz{q}{\sigma})\simeq
(C(V',-a^2b{\cdot}q'),J_{-a^2b{\cdot}q'}^{-\ve\sigma}),$$ 
$$(C_0(V,q),\jz{q}{\sigma})\simeq (C(V',-b{\cdot}q'),J_{-b{\cdot}q'}^{-\ve\sigma}).$$
The quadratic spaces $(V',-b{\cdot}q')$ and $(V',-b a^2{\cdot}q')$ are
isometric. 
We thus have an isomorphism: 
$C(V',-b{\cdot}q') \simeq C(V',-a^2b{\cdot}q')$. 
This isomorphism which is induced by the isometry between  
$(V',-b{\cdot}q')$ and $(V',-b a^2{\cdot}q')$ is compatible with 
 $J_{-b a^2{\cdot}q'}^{\ve\sigma}$ and $J_{-b{\cdot}q'}^{\ve\sigma}$, the 
proof is thus achieved.
\end{proof}

The analogue of Proposition \ref{iso-cliff-paire-inv} for the quadratic forms of 
even dimension in the particular case where $\sigma=\id$ or $-\id$
is known: in \cite{lewis-period}, David Lewis proved that for a  
quadratic space $(V_0,q_0)$ of even dimension and for any quadratic form 
$(V,q)$, there exists an isomorphism of algebras with involution
\begin{equation}
(C(V_0\perp V),J_{q_0\perp q}^{\pm\id})\simeq 
(C(V_0,q_0),J_{q_0}^{\pm\id})\otimes(C(V,-\d q_0{\cdot}q),J_{-\d
  q_0{\cdot}q}^{\mp\id}).
\end{equation}
This result can be generalized without difficulty to the case of an arbitrary 
 orthogonal symmetry:

\begin{prop}\label{iso-cliff-inv}
{\rm (\cite[Prop. 2]{lewis-period})} Let $(V_0,q_0)$ be a quadratic space of even dimension, let $(V,q)$
be an arbitrary quadratic space, let $\sigma_0$ be an orthogonal
symmetry of 
  $(V_0,q_0)$ and let $\sigma$ be an orthogonal symmetry
 of $(V,q)$. 
Suppose that $\dim
  {V_0}^-=s$. 
Then:
\begin{itemize}
\item[\rm{(a)}] If $\dim V_0\equiv 2 \mod 4$ then we have:
\begin{equation*}(C(V_0\perp V,q_0\perp q),J_{q_0\perp q}^{\sigma_0\oplus\sigma})\simeq 
(C(V_0,q_0)\otimes C(V,\d q_0{\cdot}q), J_{q_0}^{\sigma_0}\otimes 
J_{\d q_0{\cdot}q}^{(-1)^{s+1}\sigma}),\end{equation*}
in particular if $\sigma_0=\tau$ is a reflection then we have:
\begin{equation*}(C(V_0\perp V,q_0\perp q),J_{q_0\perp q}^{\tau\oplus\sigma})\simeq 
(C(V_0,q_0)\otimes C(V,\d q_0{\cdot}q), J_{q_0}^{\tau}\otimes 
J_{\d q_0{\cdot}q}^{\sigma}).\end{equation*}
\item[\rm{(b)}] If  $\dim V_0\equiv 0 \mod 4$ then we have:
\begin{equation*}(C(V_0\perp V,q_0\perp q),J_{q_0\perp q}^{\sigma_0\oplus\sigma})\simeq 
(C(V_0,q_0)\otimes C(V,\d q_0{\cdot}q), J_{q_0}^{\sigma_0}\otimes 
J_{\d q_0{\cdot}q}^{(-1)^s\sigma}),\end{equation*}
in particular if $\sigma_0=\tau$ is a reflection then we have:
\begin{equation*}(C(V_0\perp V,q_0\perp q),J_{q_0\perp q}^{\tau\oplus\sigma})\simeq 
(C(V_0,q_0)\otimes C(V,\d q_0{\cdot}q), J_{q_0}^{\tau}\otimes 
J_{\d q_0{\cdot}q}^{-\sigma}).\end{equation*}
\end{itemize} 
\end{prop}

\begin{proof}
Let $\{e_1,\cdots,e_r\}$ be an orthogonal basis of ${V_0}^+$ and 
let $\{f_1,\cdots,f_s\}$ be an orthogonal basis of ${V_0}^-$. 
Consider the element
$$z=i_{q_0}(e_1)\cdots i_{q_0}(e_r)i_{q_0}(f_1)\cdots i_{q_0}(f_s)\in
C(V_0\perp V,q_0\perp q).$$
Let $\phi:V_0\perp V_1\rightarrow C(V_0,q_0)\otimes C(V,\d
q_0{\cdot}q)$ be
the map defined by 
\begin{equation}\label{def.phi}\phi(x_0\oplus x)=i_{q_0}(x_0)\otimes 1+z^{-1}\otimes i_{\d
  q_0{\cdot}q}(x).\end{equation} 
We note that $\phi$ is a Clifford map, because:\\

$\begin{array}{rl}
\phi(x_0\oplus x)^2=&(i_{q_0}(x_0)\otimes 1+z^{-1}\otimes 
i_{\d q_0{\cdot}q}(x))^2\\
=&q_0(x_0)\otimes 1+(\d q_0)^{-1}\otimes(\d q_0{\cdot}q(x))\\
&+i_{q_0}(x_0)z^{-1}\otimes i_{\d q_0{\cdot}q}(x)+
z^{-1}i_{q_0}(x_0)\otimes i_{\d q_0{\cdot}q}(x)\\
=&(q_0\perp q)(x_0\oplus x)\cdot(1\otimes 1).
\end{array}$ \\

The map $\phi$ can therefore be extended to a homomorphism 
$${\Phi}:C(V_0\perp V)\rightarrow C(V_0,q_0)\otimes  C(V,\d q_0{\cdot}q).$$
The definition of $\phi$ in (\ref{def.phi}) implies that $\phi$ is surjective. 
As $$\dim_K C(q_0\perp q)=\dim_K(C(V_0,q_0)\otimes C(V,\d q_0\cdot
q)),$$ ${\Phi}$ is injective hence it is an isomorphism. 
It is enough to show that ${\Phi}$ is compatible with the indicated involutions. 

If
$\dim V_0\equiv 2 \mod 4$, using \ref{lemme-jz} we obtain\\

$\begin{array}{ll}
(J_q^{\sigma_0}\otimes J_{\d q_0{\cdot}q}^{(-1)^{s+1}\sigma})\circ{\Phi}
(i_{q_0\perp q}(x_0\oplus x))&=
(J_q^{\sigma_0}\otimes J_{\d
  q_0{\cdot}q}^{(-1)^{s+1}\sigma})(i_{q_0}(x_0)\otimes 1+z^{-1}\otimes i_{\d
  q_0{\cdot}q}(x))\\
&=
i_{q_0}(\sigma_0(x_0))\otimes 1+J_{q_0}^{\sigma_0}(z^{-1})\otimes 
i_{\d q_0{\cdot}q}((-1)^{s+1}\sigma(x))\\
&=i_{q_0}(\sigma_0(x_0))\otimes 1
+z^{-1}\otimes i_{\d q_0{\cdot}q}(\sigma(x)).\end{array}$\\ 

On the other hand, 
${\Phi}\circ J_{q_0\perp q}^{\sigma_0\oplus\sigma}
(i_{q_0\perp q}(x_0\oplus x))= i_{q_0}(\sigma_0(x_0))\otimes 1+
z^{-1}\otimes i_{\d q_0{\cdot}q}(\sigma(x))$. 
We therefore have:
$${\Phi}\circ J_{q_0\perp q}^{\sigma_0\oplus\sigma}=
(J_q^{\sigma_0}\otimes J_{\d
  q_0{\cdot}q}^{(-1)^{s+1}\sigma})\circ{\Phi}.$$
  
Consequently ${\Phi}$
is an isomorphism of algebras with involution. 

The proof for the case where $\dim V_0\equiv 0
\mod 4$ is similar.
\end{proof}

\begin{cor}\label{cor-cliff-inv}
Keeping the same hypotheses as in \ref{iso-cliff-inv} :
\end{cor}
\begin{itemize}
\item[\rm{(a)}] If $\dim V_0\equiv 2 \mod 4$ then we have:
\begin{equation*}
(C(V_0,q_0)\otimes C(V,q), J_{q_0}^{\sigma_0}\otimes 
J_{q}^{\sigma})\simeq 
(C(V_0\perp V,q_0\perp\d q_0{\cdot}q),
J_{q_0\perp\d q_0{\cdot}q}^{\sigma_0\oplus(-1)^{s+1}\sigma}),
\end{equation*}
in particular if $\sigma_0=\tau$ is a reflection then we have:
\begin{equation*}
(C(V_0,q_0)\otimes C(V,q), J_{q_0}^{\tau}\otimes 
J_{q}^{\sigma})\simeq 
(C(V_0\perp V,q_0\perp\d q_0{\cdot}q),
J_{q_0\perp\d q_0{\cdot}q}^{\tau\oplus\sigma}).
\end{equation*}
\item[\rm{(b)}] If  $\dim V_0\equiv 0 \mod 4$ then we have:
\begin{equation*}
(C(V_0,q_0)\otimes C(V,q), J_{q_0}^{\sigma_0}\otimes 
J_{q}^{\sigma})\simeq 
(C(V_0\perp V,q_0\perp\d q_0{\cdot}q),
J_{q_0\perp\d q_0{\cdot}q}^{\sigma_0\oplus(-1)^{s}\sigma}),
\end{equation*}
in particular if $\sigma_0=\tau$ is a reflection then we have:
\begin{equation*}
(C(V_0,q_0)\otimes C(V,q), J_{q_0}^{\tau}\otimes 
J_{q}^{\sigma})\simeq 
(C(V_0\perp V,q_0\perp\d q_0{\cdot}q),
J_{q_0\perp\d q_0{\cdot}q}^{\tau\oplus-\sigma}).
\end{equation*}
\end{itemize}

\begin{remark}\label{remark-cortella}
Anne Cortella has pointed out to me that using the notion of
determinant can lead to a simplification of the statements of \ref{iso-cliff-paire-inv} and 
\ref{iso-cliff-inv}. 
For an orthogonal symmetry
$s$ of a quadratic space $(V,\varphi)$ of dimension $n$, we define 
$\d s=(-1)^{n(n-1)/2}\det(s)$, then one can express the  
isomorphisms \ref{iso-cliff-paire-inv} and
\ref{iso-cliff-inv} in the following way:
$$(C_0(V_1\perp V,q_1\perp q),\jz{q_1\perp q}
{\sigma_1\oplus\sigma})\simeq 
(C_0(V_1,q_1)\otimes C(V,-\d q_1{\cdot}q), \jz{q_1}{\sigma_1}\otimes 
J_{-\d q_1{\cdot}q}^{-\d\sigma_1{\cdot}\sigma}),$$
 $$(C(V_0\perp V,q_0\perp q),J_{q_0\perp q}^{\sigma_0\oplus\sigma})\simeq 
(C(V_0,q_0)\otimes C(V,\d q_0{\cdot}q), J_{q_0}^{\sigma_0}\otimes 
J_{\d q_0{\cdot}q}^{\d\sigma_0{\cdot}\sigma}),$$
because for every orthogonal symmetry $s$ of a 
quadratic space $(V,\varphi)$ we have: $\det(s)=(-1)^m$ where $m$
is the dimension of the subspace of the anti-symmetric
elements of $V$ with respect to $s$.
\end{remark}
 
\section{Tensor products of quaternion algebras with involution }
\begin{prop}\label{qua-cliff-sym}
Let $(Q,J)$ be a quaternion algebra with involution over a field $K$. 
Suppose that $J$ is of the first kind.
Then there exists a nondegenerate quadratic space $(V,q)$ of dimension
$2$ and an orthogonal symmetry $\sigma:V\rightarrow V$ such that 
$(Q,J)\simeq (C(V,q),J_q^\sigma)$. 
More precisely 
\begin{itemize}
\item[(a)] if $J$ is symplectic, one can choose a quadratic space
$(V,q)$ of dimension $2$ such that $(Q,J)\simeq
(C(V,q),J_q^{-\id})$. 
\item[(b)] if $J$ is orthogonal, one can choose a quadratic space
$(V,q)$ of dimension $2$ such that $(Q,J)\simeq
(C(V,q),J_q^{\id})$.
\item[(c)] if $J$ is orthogonal, one can also choose a quadratic space
$(V,q)$ of dimension $2$ and a reflection $\sigma$ of $(V,q)$ such that $(Q,\sigma)\simeq (C(V,q),J_q^\sigma)$.
\end{itemize}
\end{prop}

\begin{proof}
Suppose that $Q=(a,b)_K$ is generated by the elements $i$ and $j$ with
$i^2=a\in K^*$, $j^2=b\in K^*$ and $i j=-j i$. 

First consider the case where $J$ is symplectic.
We have $J(i)=-i$ and $J(j)=-j$.
Consider the vector space $V=K i\oplus K j$ and the nondegenerate
quadratic form $q:V\rightarrow K$ defined by  
$q(\lambda_1 i+\lambda_2 j)=\lambda_1^2 a+\lambda_2^2 b$
for all $\lambda_1,\ \lambda_2\in K$.
We have $(Q,J)\simeq (C(V,q),J_q^{-\id})$.

Now suppose that $J$ is orthogonal. Let $u$ be an anti-symmetric
invertible element of $Q$ with respect to $J$.
The involution $J'=\intt(u)\circ J$ is of symplectic type.
According to \cite[Ch. 8, 10.1]{scharlau}, we have:
\begin{equation}\label{xu-1}
 x u^{-1} J(x) u\in K \text{ for every }x\in Q
\end{equation}
By putting $x=u$ in (\ref{xu-1}) we obtain $u^2\in K$. 
Consider the quadratic extension $K(u)/K$. 
The restriction $J|_{K(u)}$ is the nontrivial automorphism 
of $K(u)/K$. 
According to Skolem-Noether'{\bf s} Theorem, there exists an 
invertible element $v\in Q$ such that $u v=- v u$. 
As $v^2$ commutes
with both $u$ and $v$, it is in the center, i.e., $v^2\in K$. 

By putting $x=v$ in (\ref{xu-1}) we obtain $v J(v)\in K$. 
Thus there exists $\alpha\in K$ such that $J(v)=\alpha v$. 
We have $v=J^2(v)=\alpha^2 v$. 
Thus $\alpha=1$ or  $\alpha=-1$. 
The case
$\alpha=-1$ is excluded, because $J$ is orthogonal
(the dimension of the subspace of anti-symmetric elements of
$Q$ is $1$). 
Thus we have $J(v)=v$. The 
elements $u$ and $v$ satisfy: $u^2=a'\in K$,  $v^2=b'\in K$,
$J(u)=-u$, $J(v)=v$ and $u v+ v u=0$. 
Consider the vector space  $V=K u\oplus K v$ and
 $q:V\rightarrow K$ defined by 
$q(\lambda_1 u+\lambda_2 v)=\lambda_1^2 a'+\lambda_2^2b'$
 for
every $\lambda_1,\lambda_2\in K$ and the orthogonal symmetry 
$\sigma:V\rightarrow V$ defined by $\sigma(u)=-u$ and
$\sigma(v)=v$. 
Consider the   
$K$-linear map $f:V\rightarrow Q$, defined by $f(u)=-u$,
$f(v)=v$. 
The map $f$ is a Clifford map and can be extended 
to an isomorphism between $C(V,q)$ and $Q$. 
The construction of $f$ implies that it is compatible with $J_q^\sigma$ and $J$.
It is therefore an isomorphism of algebras with involution. 
In this case $\sigma$ is a reflection of $(V,q)$ because the dimension of the vector space of anti-symmetric elements of 
$V$ with respect to $\sigma$ is $1$.

Let $w=uv$, we have $w^2=-a'b'$ and $J(w)=w$.
Consider the vector space 
$V=K w\oplus K v$, the quadratic form $q:V\rightarrow K$ defined by 
$q(\lambda_1 w+\lambda_2 v)=-\lambda_1^2 a'b'+\lambda_2^2b'$ for
all $\lambda_1,\lambda_2\in K$ and the orthogonal symmetry
$\sigma:V\rightarrow V$ defined by $\sigma(w)=w$ and 
$\sigma(v)=v$.
Consider the 
$K$-linear map $f:V\rightarrow Q$ defined by $f=\id|_V$.
As $f$ is a Clifford map, it can be extended to an isomorphism 
between $C(V,q)$ and $Q$. The construction of
$f$ shows that $f$ is compatible with $J_q^\sigma$ and $J$.
It is therefore an isomorphism of algebras with involution. 
In this case, $\sigma$ is the identity map.
\end{proof}

As a direct consequence of the proof of Proposition
\ref{qua-cliff-sym}, we obtain

\begin{cor}
Let $a$ and $b$ be two invertible elements of a field $K$. 
Then we have an isomorphism 
$$(C(\langle
a,b\rangle),J^{-+})\simeq(C(\langle-ab,b\rangle),J^{++}),$$
or equivalently $$(C(\langle
a,b\rangle),J^{++})\simeq(C(\langle-ab,b\rangle),J^{-+}),$$
here $J^{-+}$ is the involution of $C(\langle a,b\rangle)$ induced by
the reflection $\tau=-\id_{K\cdot x}\oplus \id_{K\cdot y}$, where
$V=K\cdot x\perp K\cdot y$ is the underlying vector space of $q=\langle
a, b\rangle$ with $q(x)=a$, $q(y)=b$ and $J^{++}$ is the involution of
$q'=\langle -ab,b\rangle$ induced by the identity map on the
underlying vector space of $q'$.    
\end{cor}

\begin{thm}\label{mul-qua-cliff-sym}
Let $(Q_1,J_1),\ \cdots,(Q_n,J_n)$ be quaternion algebras over a
field $K$ with involutions of the first kind.
Then there exists a quadratic space $(V,q)$ over $K$ of dimension $2n$ and 
an orthogonal symmetry $\sigma:V\rightarrow V$ such that  
$$(C(V,q),J_q^\sigma)\simeq (Q_1,J_1)\otimes\cdots\otimes(Q_n,J_n).$$
\end{thm}

\begin{proof}
We prove the result by induction on $n$. If  
$n=1$, we use directly Proposition \ref{qua-cliff-sym}. 

Assume that $n>1$. 
By induction hypothesis,
there exists a quadratic space $(W,q)$  of dimension $2(n-1)$
over $K$ and an orthogonal symmetry $\sigma:W\rightarrow W$ such that 
$$(C(W,q),J_{q}^{\sigma})\simeq
(Q_2,J_2)\otimes\cdots\otimes(Q_{n-1},J_{n-1}).$$

According to Proposition \ref{qua-cliff-sym}, there also exists a quadratic space
$(W_0,q_0)$ of dimension $2$ over $K$ and an orthogonal symmetry 
$\sigma_0:W_0\rightarrow W_0$ such that 
$$(C(W_0,q_0),J_{q_0}^{\sigma_0})\simeq (Q_1,J_1).$$

Let $s$ be the dimension of the anti-symmetric elements of  
$(W_0,q_0)$. We have obviously $s=2$ if $J_1$ is symplectic. 
Using 
\ref{iso-cliff-inv}, we obtain
\begin{equation*}
\begin{array}{lll}
(Q_1,J_1)\otimes\cdots\otimes(Q_n,J_n)&\simeq& 
(C(W_0,q_0),J_{q_0}^{\sigma_0})\otimes(C(W,q),J_{q}^{\sigma})\\
&\simeq& (C(W_0\perp
W,h),
J_{h}^{\tau}),
\end{array}
\end{equation*}

where $h=q_0\perp(\d q_0)^{-1}\cdot q$ and $\tau=\sigma_0\oplus(-1)^{-(s+1)}\sigma$.
\end{proof}

According to a result due to Albert
(cf. \cite[16.1]{boi}),
 every central simple algebra with involution $A$
of degree $4$ can be decomposed as a tensor product of two
quaternion algebras.
In \cite{kps}, it has been shown that a central simple algebra $A$
of degree $4$ with an orthogonal involution $\sigma$ can be decomposed as a tensor product
of two quaternion algebras with symplectic involution if and only if
it can be decomposed as a tensor product of two quaternion algebras
with orthogonal involution if and only if the discriminant of $\sigma$
is trivial. 
See also \cite[15.12]{boi}.
Here we complement these results by showing that:  

\begin{lem}\rm{(\cite{kps})}\label{biquaternions-orthogonal}
Let $(A,\sigma)$ be a central simple algebra of degree $4$ over a
field $K$. 
The following assertions are equivalent: 
\begin{itemize}
\item[(i)] $(A,\sigma)$ is isomorphic to the tensor product of two
  quaternion algebras with ortho\-gonal involutions.
\item[(ii)] $(A,\sigma)$ is isomorphic to the tensor product of two
  quaternion algebras with symplectic involutions.
\item[(iii)] there exists a quadratic space $(V,q)$ of dimension $4$
  over $K$ and a reflection $\tau$ of $(V,q)$ such that $(A,\sigma)\simeq(C(V,q),J_{q}^{\tau}).$
\end{itemize}
\end{lem}

\begin{proof}
For the sake of completeness, we also prove the equivalence of (i) and
(ii) using  \ref{iso-cliff-inv}.

(i)$\Rightarrow$(ii). Let $(A,\tau)=(Q_1,\tau_1)\otimes(Q_2,\tau_2)$ where 
$(Q_1,\tau_1)$ and $(Q_2,\tau_2)$ are quaternion algebras 
and $\tau_1$ and $\tau_2$ are orthogonal. 
According to \ref{qua-cliff-sym},
there exist quadratic spaces $(V_1,q_1)$ and
$(V_2,q_2)$ of dimension $2$ such that:
 $$(Q_1,\tau_1)\simeq (C(V_1,q_1),J_{q_1}^{id}),$$ 
$$(Q_2,\tau_2)\simeq (C(V_2,q_2),J_{q_2}^{\id}).$$
 
We thus obtain
\begin{equation*}
\begin{array}{lll}
(Q_1,\tau_1)\otimes(Q_2,\tau_2)&\simeq &
(C(V_1,q_1),J_{q_1}^{\id})\otimes(C(V_2,q_2),J_{q_2}^{\id})\\
\text{using \ref{cor-cliff-inv}}&\simeq&(C(V_1\perp V_2,q_1\perp(\d q_1)\cdot q_2),J_{q_1\perp(\d
  q_1)\cdot q_2}^{\id\oplus-\id})\\
\text{using \ref{iso-cliff-inv}}&\simeq&(C(V_1,\d q_2\cdot q_1),J_{\d
  q_2\cdot q_1}^{-\id})\otimes(C(V_2,\d q_1\cdot q_2),J_{\d q_1\cdot q_2}^{-\id})
\end{array}
\end{equation*}
The involutions $J_{\d q_2\cdot q_1}^{-\id}$ and $J_{\d q_1\cdot q_2}^{-\id}$ are
both symplectic, the proof is therefore achieved.

The proof of (ii)$\Rightarrow$(i) is similar.

(i)$\Rightarrow$(iii). According to \ref{qua-cliff-sym}, there exist
quadratic spaces $(V_1,q_1)$ and $(V_2,q_2)$ and a reflection $\rho$
of $(V_2,q_2)$ such that $$(Q_1,\tau_1)\simeq
(C(V_1,q_1),J_{q_1}^{id}),$$
$$(Q_2,\tau_2)\simeq (C(V_2,q_2),J_{q_2}^{\rho}).$$

We therefore obtain
\begin{equation*}
\begin{array}{lll}
(Q_1,\tau_1)\otimes(Q_2,\tau_2)&\simeq &
(C(V_1,q_1),J_{q_1}^{\id})\otimes(C(V_2,q_2),J_{q_2}^{\rho})\\
\text{using \ref{cor-cliff-inv}}&\simeq&(C(V_1\perp V_2,q_1\perp(\d q_1)\cdot q_2),J_{q_1\perp(\d
  q_1)\cdot q_2}^{\id\oplus-\rho})
\end{array}
\end{equation*}
It suffices to put $V=V_1\perp V_2$, $q=q_1\perp(\d q_1)\cdot q_2$ and
$\tau=\id\oplus-\rho$, we then have $(A,\sigma)\simeq(C(V,q),J_{q}^{\tau}).$

The implication (i)$\Rightarrow$(iii) follows from \ref{qua-cliff-sym} and \ref{iso-cliff-inv}.
\end{proof}

In \cite{rowen}, it has been shown that every division algebra
of degree $4$ with symplectic involution can be decomposed as a tensor
product of two quaternion algebras with involution. 
More generally if $A$ is a central simple algebra with a symplectic
involution can be decomposed as a tensor product of two quaternion
algebras with involution, see \cite[Thm. 10.5, Prop. 10.21]{shapiro}.
We complement these results by showing that:

\begin{lem}\label{biquaternions-symplectic} {\rm (Compare with \cite[Lemma 10.6]{shapiro})}
Let $(A,\sigma)$ be a central simple algebra of degree $4$ over a
field $K$. 
The following assertions are equivalent: 
\begin{itemize}
\item[(i)] $(A,\sigma)$ is isomorphic to
  $(Q_1,\sigma_1)\otimes(Q_2,\sigma_2)$, where $Q_1$ and $Q_2$ are
  quaternion algebras over $K$, $\sigma_1$ is symplectic and
  $\sigma_2$ is orthogonal.
\item[(ii)] there exists a quadratic space $(V,q)$ of dimension $4$
  over $K$ such that $(A,\sigma)\simeq(C(V,q),J_{q}^{\id}).$
\item[(iii)] there exists a quadratic space $(V,q)$ of dimension $4$
  over $K$ such that $(A,\sigma)\simeq(C(V,q),J_{q}^{-\id}).$
\end{itemize}
\end{lem}

\begin{proof}
According to \ref{qua-cliff-sym}, there exists quadratic spaces
$(V_1,q_1)$ and $(V_2,q_2)$ such that
$(Q_1,\sigma_1)\simeq(C(V_1,q_1),J_{q_1}^{-\id})$ and
$(Q_2,\sigma_2)\simeq (C(V_2,q_2),J_{q_2}^{\id})$.
We thus obtain
\begin{equation*}
\begin{array}{lll}
(Q_1,\sigma_1)\otimes(Q_2,\sigma_2)&\simeq &
(C(V_1,q_1),J_{q_1}^{-\id})\otimes(C(V_2,q_2),J_{q_2}^{\id})\\
\text{using \ref{cor-cliff-inv}}&\simeq&(C(V_1\perp V_2,q_1\perp(\d q_1)\cdot q_2),J_{q_1\perp(\d
  q_1)\cdot q_2}^{-\id\oplus-\id}).
\end{array}
\end{equation*}
Similarly we have:
\begin{equation*}
\begin{array}{lll}
(Q_1,\sigma_1)\otimes(Q_2,\sigma_2)&\simeq &
(C(V_1,q_1),J_{q_1}^{-\id})\otimes(C(V_2,q_2),J_{q_2}^{\id})\\
\text{using \ref{cor-cliff-inv}}&\simeq&(C(V_1\perp V_2,(\d q_2)\cdot q_1\perp q_2),J_{(\d q_2)\cdot q_1\perp q_2}^{\id\oplus\id}).
\end{array}
\end{equation*}
We thus have the implications (i)$\Rightarrow$(ii) and
(i)$\Rightarrow$(iii). 
The implications (ii)$\Rightarrow$(i) and (iii)$\Rightarrow$(i) follow
from \ref{iso-cliff-inv}.
\end{proof}

\begin{lem}\label{3-quaternions}
Let $(Q_1,\sigma_1)$, $(Q_2,\sigma_2)$ and $(Q_3,\sigma_3)$ be
quaternion algebras with involutions of the first kind. 
Let $(A,\sigma)=(Q_1,\sigma_1)\otimes(Q_2,\sigma_2)\otimes(Q_3,\sigma_3)$.
\begin{itemize}
\item[(a)] if $\sigma$ is symplectic then there exists a quadratic
  space $(V,q)$ of dimension $6$ over $K$ such that
  $(A,\sigma)\simeq(C(V,q),J_q^{\id})$.
\item[(b)] if $\sigma$ is orthogonal then there exists a quadratic
  space $(A,\sigma)$ of dimension $6$ such that $(A,\sigma)\simeq(C(V,q),J_q^{-\id})$.
\end{itemize} 
\end{lem}

\begin{proof}
If $\sigma$ is symplectic then we may assume that either all $\sigma_i$,
$i=1,2,3$, are symplectic or, $\sigma_1$ is symplectic and $\sigma_2$
and $\sigma_3$ are orthogonal.   
Thanks to \ref{biquaternions-orthogonal}, the first case is reduced to the
second case. 
In the second case, there exist quadratic spaces $(V_1,q_1)$,
$(V_2,q_2)$ and $(V_3,q_3)$ of dimension $2$ over $K$ such that
$(C(V_1,q_1),J_{q_1}^{-\id})\simeq (Q_1,\sigma_1)$,
$(C(V_2,q_2),J_{q_2}^{\id})\simeq (Q_2,\sigma_2)$ and
$(C(V_3,q_3),J_{q_3}^{\id})\simeq (Q_3,\sigma_3)$. 
According to \ref{biquaternions-symplectic}, there exists a quadratic
space $(W,h)$ of dimension $4$ over $K$ such that
$$(C(W,h),J_{h}^{\id})\simeq (Q_1,\sigma_1)\otimes(Q_2,\sigma_2).$$
According to \ref{qua-cliff-sym}, there exists a quadratic space
$(W',h')$ of dimension $2$ over $K$ such
that: $$(C(W',h'),J_{h'}^{\id})\simeq (Q_3,\sigma_3).$$  
We thus obtain
\begin{equation*}
\begin{array}{lll}
(Q_1,\sigma_1)\otimes(Q_2,\sigma_2)\otimes(Q_3,\sigma_3)&\simeq& 
(C(W,h),J_{h}^{\id})\otimes(C(W',h'),J_{h}^{\id})\\
\textrm{using \ref{cor-cliff-inv} (b)} &\simeq& (C(W\perp
W',h\perp(\d h)\cdot h'),
J_{h\perp (\d h)\cdot h'}^{\id\oplus\id}).
\end{array}
\end{equation*}
If $\sigma$ is orthogonal then we may assume that either all $\sigma_i$,
$i=1,2,3$, are orthogonal or, $\sigma_1$ is orthogonal and $\sigma_2$
and $\sigma_3$ are symplectic. 
But thanks to \ref{biquaternions-orthogonal}, the first case is
reduced to the second one.  
In the second case, according to \ref{biquaternions-symplectic}, there exists a quadratic
space $(W,h)$ of dimension $4$ over $K$ such that
$$(C(W,h),J_{h}^{-\id})\simeq (Q_1,\sigma_1)\otimes(Q_2,\sigma_2).$$
According to \ref{qua-cliff-sym}, there exists a quadratic space
$(W',h')$ of dimension $2$ over $K$ such
that: $$(C(W',h'),J_{h'}^{-\id})\simeq (Q_3,\sigma_3).$$  
We thus obtain 
\begin{equation*}
\begin{array}{lll}
(Q_1,\sigma_1)\otimes(Q_2,\sigma_2)\otimes(Q_3,\sigma_3)&\simeq& 
(C(W,h),J_{h}^{-\id})\otimes(C(W',h'),J_{h'}^{-\id})\\
\textrm{using \ref{cor-cliff-inv} (b)} &\simeq& (C(W\perp
W',h\perp(\d h)\cdot h'),
J_{h\perp (\d h)\cdot h'}^{-\id\oplus-\id}).
\end{array}
\end{equation*}
The proof is thus achieved.
\end{proof}

\begin{lem}\label{4-quaternions}
Let $(Q_1,\sigma_1)$, $(Q_2,\sigma_2)$, $(Q_3,\sigma_3)$ and $(Q_4,\sigma_4)$ be
quaternion algebras with involutions of the first kind. 
Let $(A,\sigma)=(Q_1,\sigma_1)\otimes(Q_2,\sigma_2)\otimes(Q_3,\sigma_3)\otimes(Q_4,\sigma_4)$.
\begin{itemize}
\item[(a)] if $\sigma$ is orthogonal then there exist quadratic
  spaces $(V,q)$ and $(V',q')$ of dimension $8$ over $K$ such that
  $(A,\sigma)\simeq(C(V,q),J_q^{\id})$ and $(A,\sigma)\simeq(C(V',q'),J_{q'}^{-\id})$.
\item[(b)] if $\sigma$ is symplectic then there exists a quadratic
  space $(A,\sigma)$ of dimension $8$ over $K$ and a reflection $\tau$
  of $(V,q)$ such that $(A,\sigma)\simeq(C(V,q),J_q^{\tau})$.
\end{itemize} 
\end{lem}

\begin{proof}
(a) As $\sigma$ is orthogonal, the numbers of $\sigma_i$, $i=1,\cdots,4$, which are
symplectic, should be even.  
Thanks to \ref{biquaternions-orthogonal}, we are reduced to consider the case
where two of $\sigma_i$, $i=1,\cdots,4$, are orthogonal and two of them
are symplectic. 
Without loss of generality, it may be assumed that $\sigma_1$ and $\sigma_3$ are orthogonal and $\sigma_2$ and
$\sigma_4$ are symplectic. 
There exist so quadratic spaces $(V_i,q_i)$, $i=1,\cdots,4$, such
that $(Q_i,\sigma_i)\simeq(C(V_i,q_i),J_{q_i}^{\id})$ for $i=1,\ 3$
and $(Q_i,\sigma_i)\simeq(C(V_i,q_i),J_{q_i}^{-\id})$ for $i=2,\ 4$. 
Using \ref{biquaternions-symplectic}, there exist quadratic spaces
$(V_1,q_1)$ and $(V_2,q_2)$ of dimension $4$ over $K$ such that:
$$(Q_1,\sigma_1)\otimes(Q_2,\sigma_2)\simeq
(C(V_1,q_1),J_{q_1}^{\id}),$$
$$(Q_3,\sigma_3)\otimes(Q_3,\sigma_3)\simeq (C(V_2,q_2),J_{q_2}^{\id}).$$
We thus obtain
\begin{eqnarray*}
(A,\sigma)&\simeq&(C(V_1,q_1),J_{q_1}^{\id})\otimes(C(V_2,q_2),J_{q_2}^{\id})\\
\text{using \ref{cor-cliff-inv}
  (b)}&\simeq&(C(V_1\perp V_2,q_1\perp \d q_1\cdot q_2),J_{q_1\perp \d q_1\cdot q_2}^{\id\oplus\id}).
\end{eqnarray*}
The proof of the other assertion of (a) is similar and is left to the
reader.

(b) Using \ref{biquaternions-orthogonal}, we may assume that
$\sigma_1$, $\sigma_2$ and $\sigma_3$ are symplectic and $\sigma_4$
is orthogonal. 
According to \ref{biquaternions-orthogonal} and \ref{3-quaternions},
there exist quadratic spaces
$(V_1,q_1)$ and $(V_2,q_2)$, {respectively}, of dimension $6$ and $2$
over $K$, and a reflection $\tau'$ of $(V_2,q_2)$ such that:
$$(C(V_1,q_1),J_{q_1}^{\id})\simeq(Q_1,\sigma_1)\otimes(Q_2,\sigma_2)\otimes(Q_3,\sigma_3).$$ 
$$(C(V_2,q_2),J_{q_2}^{\tau'})\simeq(Q_1,\sigma_1)$$   
We so obtain
\begin{eqnarray*}
(A,\sigma)&\simeq&(C(V_1,q_1),J_{q_1}^{\id})\otimes(C(V_2,q_2),J_{q_2}^{\tau'})\\
\text{using \ref{cor-cliff-inv}
  (a)}&\simeq&(C(V_1\perp V_2,q_1\perp \d q_1\cdot q_2),J_{q_1\perp \d q_1\cdot q_2}^{\id\oplus-\tau'}).
\end{eqnarray*}
Therefore it suffices to put $(V,q)=(V_1\perp V_2,q_1\perp \d q_1\cdot
q_2)$ and $\tau=\id|_{V_1}\oplus-\tau'|_{V_2}$.
\end{proof}

\begin{prop}\label{tensor-odd}
Let $n$ be an odd positive integer. 
Let $(Q_1,\sigma_1),\cdots,(Q_n,\sigma_n)$ be quaternion algebras with
involution of the first kind over a field $K$ and let
$$(A,\sigma)=(Q_1,\sigma_1)\otimes\cdots\otimes(Q_n,\sigma_n).$$ 
Then: 
\begin{itemize}
\item[(a)] If $n\equiv 1 \mod 4$ and if $\sigma$ is symplectic then there exists a quadratic
  space $(V,q)$ of dimension $2n$ over $K$ such that 
$(A,\sigma)\simeq (C(V,q),J_q^{-\id})$.
\item[(b)] If $n\equiv 1 \mod 4$ and if $\sigma$ is orthogonal then there exists a quadratic
  space $(V,q)$ of dimension $2n$ over $K$ such that 
$(A,\sigma)\simeq (C(V,q),J_q^{\id})$.
\item[(c)] If $n\equiv 3 \mod 4$ and if $\sigma$ is symplectic then there exists a quadratic
  space $(V,q)$ of dimension $2n$ over $K$ such that $(A,\sigma)\simeq (C(V,q),J_q^{\id})$.
\item[(d)] If $n\equiv 3 \mod 4$ and if $\sigma$ is orthogonal then there exists a quadratic
  space $(V,q)$ of dimension $2n$ over $K$ such that $(A,\sigma)\simeq (C(V,q),J_q^{-\id})$.
\end{itemize} 
\end{prop}

\begin{proof}
We prove all the assertions by induction on $n$. 
If $n=1$, the assertions (a) and (b) follow from
\ref{qua-cliff-sym}. 
If $n=3$, the assertions (c) and (d) follow from \ref{3-quaternions}. 
Suppose that the assertions (a) and (b) and the assertions (c) and (d)
are, respectively, true for $n=4k+1$ and $n=4k+3$ where $k$ is a
nonnegative integer.
We show that the statements (a) and (b) and the statements (c) and (d)
are, respectively, true for $n=4k+5$ and $n=4k+7$. 

We first consider the assertion (a) for the case where $n=4k+5$. 
As $\sigma$ is symplectic we deduce that either all of $\sigma_i$,
$i=1,\cdots,n$, are symplectic or, at least two of $\sigma_i$,
$i=1,\cdots,n$, say $\sigma_1$ and $\sigma_2$, are orthogonal.
Using \ref{biquaternions-orthogonal}, there exist two symplectic
involution $\sigma'_1$ and $\sigma'_2$ such that
$\sigma_1\otimes\sigma_2\simeq\sigma'_1\otimes\sigma'_2$. 
The second case is so reduced to the first one. 
We may thus assume that all of $\sigma_i$, $i=1,\cdots,n$, are
symplectic.   
As by induction hypothesis, (c) is true for $n=4k+3$, there
exists a quadratic space $(W,h)$ of dimension $8k+6$ over $K$ such
that $$(Q_3,\sigma_3)\otimes\cdots\otimes(Q_n,\sigma_n)\simeq (C(W,h),J_{h}^{\id}).$$ 
According to \ref{qua-cliff-sym}, there exist two quadratic spaces
$(V_1,q_1)$ such that
$(Q_1,\sigma_1)\simeq(C(V_1,q_1),J_{q_1}^{-\id})$ and
$(Q_2,\sigma_2)\simeq(C(V_2,q_2),J_{q_2}^{-\id})$. 
We thus obtain
\begin{equation*}
\begin{array}{lll}
(Q_1,\sigma_1)\otimes\cdots\otimes(Q_n,\sigma_n)&\simeq& 
(Q_1,\sigma_1)\otimes(Q_2,\sigma_2)\otimes(C(W,h),J_{h}^{\id})\\
\textrm{using \ref{cor-cliff-inv} (a)} &\simeq&
(Q_1,\sigma_1)\otimes(C(V_2\perp W,q_2\perp\d q_2\cdot
h),J_{q_2\perp\d q_2\cdot h}^{-\id\oplus-\id})\\
\textrm{using \ref{cor-cliff-inv} (b)} &\simeq&
(C(V_1\perp V_2\perp W,\d q'\cdot q_1\perp q'),J_{\d q'\cdot q_1\perp q'}^{-\id\oplus-\id\oplus-\id}),
\end{array}
\end{equation*}
where $q'=q_2\perp\d q_2\cdot h$.
The quadratic space $(V,q)=(V_1\perp V_2\perp W,\d q'\cdot q_1\perp
q')$ is indeed the one we were looking for.

We now prove the assertion (b) for the case where $n=4k+5$. 
By the same argument as before, we may assume that
all of $\sigma_i$, $i=1,\cdots,n$, are orthogonal. 
As by induction hypothesis (d) is true for $n=4k+3$, there exists
a quadratic space $(W,h)$ of dimension $8k+6$ over $K$ such that:
$$(Q_3,\sigma_3)\otimes\cdots\otimes(Q_n,\sigma_n)\simeq (C(W,h),J_{h}^{-\id}).$$   
According to \ref{qua-cliff-sym}, there exist quadratic spaces
$(V_1,q_1)$ and $(V_2,q_2)$ such that
$(Q_1,\sigma_1)\simeq(C(V_1,q_1),J_{q_1}^{\id})$ and
$(Q_2,\sigma_2)\simeq(C(V_2,q_2),J_{q_2}^{\id})$. 
We so obtain
\begin{equation*}
\begin{array}{lll}
(Q_1,\sigma_1)\otimes\cdots\otimes(Q_n,\sigma_n)&\simeq& 
(Q_1,\sigma_1)\otimes(Q_2,\sigma_2)\otimes(C(W,h),J_{h}^{-\id})\\
\textrm{using \ref{cor-cliff-inv} (a)} &\simeq&
(Q_1,\sigma_1)\otimes(C(V_2\perp W,q_2\perp\d q_2\cdot
h),J_{q_2\perp\d q_2\cdot h}^{\id\oplus\id})\\
\textrm{using \ref{cor-cliff-inv} (b)} &\simeq&
(C(V_1\perp V_2\perp W,\d q'\cdot q_1\perp q'),J_{\d q'\cdot q_1\perp q'}^{\id\oplus\id\oplus\id}),
\end{array}
\end{equation*}
where $q'=q_2\perp\d q_2\cdot h$.
The quadratic space $(V,q)=(V_1\perp V_2\perp W,\d q'\cdot q_1\perp
q')$ is indeed the one we were looking for.

We now prove the assertion (c) for the case where $n=4k+7$. 
As $\sigma$ is symplectic, by the same argument as before, we may assume that all of $\sigma_i$ are symplectic.  
As we have already shown that the assertion (a) is true for $n=4k+5$, there
exists so a quadratic space $(W,h)$ of dimension $8k+10$ such that
$$(Q_3,\sigma_3)\otimes\cdots\otimes(Q_n,\sigma_n)\simeq
(C(W,h),J_{h}^{-\id}).$$
According to \ref{qua-cliff-sym}, there exist quadratic spaces
$(V_1,q_1)$ and $(V_2,q_2)$ such that
$(Q_1,\sigma_1)\simeq(C(V_1,q_1),J_{q_1}^{-\id})$ and
$(Q_2,\sigma_2)\simeq(C(V_2,q_2),J_{q_2}^{-\id})$. 
We thus obtain
\begin{equation*}
\begin{array}{lll}
(Q_1,\sigma_1)\otimes\cdots\otimes(Q_n,\sigma_n)&\simeq& 
(Q_1,\sigma_1)\otimes(Q_2,\sigma_2)\otimes(C(W,h),J_{h}^{\id})\\
\textrm{using \ref{cor-cliff-inv} (a)} &\simeq&
(Q_1,\sigma_1)\otimes(C(V_2\perp W,q_2\perp\d q_2\cdot
h),J_{q_2\perp\d q_2\cdot h}^{-\id\oplus-\id})\\
\textrm{using \ref{cor-cliff-inv} (b)} &\simeq&
(C(V_1\perp V_2\perp W,\d q'\cdot q_1\perp q'),J_{\d q'\cdot q_1\perp q'}^{-\id\oplus-\id\oplus-\id}),
\end{array}
\end{equation*}
where $q'=q_2\perp\d q_2\cdot h$.

The proof of (d) is similar and is left to the reader.
\end{proof}

\begin{prop}\label{tensor-even}
Let $n$ be an even positive integer. 
Let $(Q_1,\sigma_1),\cdots,(Q_n,\sigma_n)$ be quaternion algebras with
involution of the first kind over a field $K$ and let
$$(A,\sigma)=(Q_1,\sigma_1)\otimes\cdots\otimes(Q_n,\sigma_n).$$ 
Then: 

\begin{itemize}
\item[(a)] If $n\equiv 0\mod 4$ and if $\sigma$ is orthogonal then
  there exist quadratic spaces $(V,q)$ and $(V',q')$ of dimension $2n$ over $K$
  such that $(A,\sigma)\simeq(C(V,q),J_q^{\id})$ and $(A,\sigma)\simeq(C(V',q'),J_{q'}^{-\id})$.
\item[(b)] If $n\equiv 0\mod 4$ and if $\sigma$ is symplectic then
  there exists a quadratic space $(V,q)$ of dimension $2n$ over $K$ and
  a reflection $\tau$ of $(V,q)$ such that $(A,\sigma)\simeq(C(V,q),J_q^{\tau})$.
\item[(c)] If $n\equiv 2\mod 4$ and if $\sigma$ is symplectic then
  there exist quadratic spaces $(V,q)$ and $(V',q')$ of dimension $2n$ over $K$
  such that $(A,\sigma)\simeq(C(V,q),J_q^{\id})$ and $(A,\sigma)\simeq(C(V',q'),J_{q'}^{-\id})$.
\item[(d)] If $n\equiv 2\mod 4$ and if $\sigma$ is orthogonal then
  there exists a quadratic space $(V,q)$ of dimension $2n$ over $K$ and
  a reflection $\tau$ of $(V,q)$ such that $(A,\sigma)\simeq(C(V,q),J_q^{\tau})$.
\end{itemize}
\end{prop}

\begin{proof}
We prove all assertions by induction on $n$. 
According to \ref{biquaternions-orthogonal},
\ref{biquaternions-symplectic} and \ref{4-quaternions}, the assertions
(a) and (b) are true for $n=4$ and (c) and (d) are true
for $n=2$.  
Suppose that the assertions (a) and (b) are true for $n=4k+4$ and (c) and (d) are true for $n=4k+2$ where $k$ is a
nonnegative integer. 
We should prove that the statements (a) and (b) are true for $n=4k+8$
and (c) and (d) are true for $n=4k+6$.    

We prove the assertion (a) for $n=4k+8$. 
As $\sigma$ is orthogonal, the number of the involutions $\sigma_i$,
$i=1,\cdots,n$, which are symplectic are even. 
Using \ref{biquaternions-orthogonal}, we may suppose that all of
$\sigma_i$, $i=1,\cdots,n$, are orthogonal. 
By induction  hypothesis there exists a quadratic space $(V',q')$
of dimension $8k+8$ over $K$ such
that $$(Q_5,\sigma_5)\otimes\cdots\otimes(Q_n,\sigma_n)\simeq(C(V',q'),J_{q'}^{\id}).$$ 
Using \ref{4-quaternions}, there exists a quadratic form $(V'',q'')$ of
dimension $8$ over $K$ such
that 
\begin{equation}
  \label{eq:Q1-Q4}
(Q_1,\sigma_1)\otimes\cdots\otimes(Q_4,\sigma_4)\simeq(C(V'',q''),J_{q''}^{\id}).
\end{equation}
We thus obtain
\begin{eqnarray*}
 (A,\sigma)&\simeq&(C(V'',q''),J_{q''}^{\id})\otimes(C(V',q'),J_{q'}^{\id})\\
\text{using \ref{cor-cliff-inv} (b)}&\simeq&(C(V''\perp V',q''\perp\d
q''\cdot q'), J_{q''\perp\d q''\cdot q'}^{\id\oplus\id}).
\end{eqnarray*}
It suffices to put $(V,q)=(V''\perp V',q''\perp\d q''\cdot q')$.
The proof of the second assertion of (a) is similar. 

In order to prove the assertion (b) for $n=4k+8$, note that as $\sigma$ is symplectic, using \ref{biquaternions-orthogonal}
 we may suppose that $\sigma_n$ is symplectic and all $\sigma_i$, for
 $i=1,\cdots,n-1$, are orthogonal. 
By induction hypothesis, there exists a quadratic space
$(V',q')$ of dimension $8k+8$ over $K$ and a reflection $\tau'$ of $(V',q')$ such
that $$(Q_5,\sigma_5)\otimes\cdots\otimes(Q_n,\sigma_n)\simeq(C(V',q'),J_{q'}^{\tau'}).$$ 
The relation (\ref{eq:Q1-Q4}) is also satisfied. 
We thus obtain 
\begin{eqnarray*}
 (A,\sigma)&\simeq&(C(V'',q''),J_{q''}^{\id})\otimes(C(V',q'),J_{q'}^{\tau'})\\
\text{using \ref{cor-cliff-inv} (b)}&\simeq&(C(V''\perp V',q''\perp\d
q''\cdot q'), J_{q''\perp\d q''\cdot q'}^{\id\oplus\tau'}).
\end{eqnarray*}
It suffices to put $(V,q)=(V''\perp V',q''\perp\d q''\cdot q')$ and $\tau=\id|_{V''}\oplus\tau'$.

The proof of the assertions (c) and (d) are similar and are left to
the reader.
\end{proof}

\subsection{Involutions of the second kind}

\begin{lem} \label{Albert-generalization}{\rm (Albert).}
Let $(Q,\sigma)$ be a quaternion algebra with involution over a field $K$. 
The following assertions are equivalent:
\begin{itemize}
\item[(i)] The $\sigma$ involution $\sigma$ is of the second kind, in
  other words $\sigma|_K$ is a nontrivial automorphism of $K$. 
\item[(ii)] There exists a subfield $k$ of $K$ with $[K:k]=2$ and a
  quaternion algebra $Q_0$ over $k$ such that $(Q,\sigma)\simeq
  (Q_0\otimes_kK,\gamma\otimes-)$ where $\gamma$ is the canonical
  involution of $Q_0$ and $-:K\rightarrow K$ is the nontrivial
  automorphism of $K/k$.    
\item[(iii)] There exists a subfield $k$ of $K$ with $[K:k]=2$ and a
  quaternion algebra $Q_0$ over such that $(Q,\sigma)\simeq
  (Q_0\otimes_kK,\rho\otimes-)$ where $\rho$ is an orthogonal
  involution of $Q_0$ and $-:K\rightarrow K$ is the nontrivial
  automorphism of $K/k$.    
\item[(iv)] There exists a quadratic space $(V,q)$ of dimension $3$
  over $K$ such that $(Q,\sigma)\simeq(C(V,q),J_q^{\id})$.
\end{itemize}
\end{lem}

\begin{proof}
Our sole contribution is to prove the equivalence of (ii), (iii) and (iv). 
According to \ref{qua-cliff-sym}, there exists a quadratic space
$(V_0,q_0)$ of dimension $2$ with $q_0\simeq\langle a,b\rangle$ such that
$(Q_0,\gamma)\simeq(C(V_0,q_0),J_{q_0}^{-\id})$.
We may assume that $K=k(\sqrt{c})$ where $c\in K$ in a non-square
element. 
We have $(K,-)\simeq(C(\langle c\rangle),J^{-\id})$.  
We thus obtain
\begin{eqnarray*}
(Q_0\otimes_kK,\gamma\otimes-)&\simeq&
(C(V_0,q),J_{q_0}^{-\id}))\otimes(C(\langle c\rangle),J^{-\id})\\
\text{using \ref{cor-cliff-inv} (a)}&\simeq&(C(\langle
a,b,-abc\rangle),J^{--+}) \\
\text{using \ref{iso-cliff-inv} (a)}&\simeq&(C(\langle
c\rangle),J^{-})\otimes(C(\langle b,-abc\rangle),J^{-+})\\
\text{using \ref{qua-cliff-sym}, $\exists d,\ e\in k^{\times}$}&\simeq&(C(\langle
c\rangle),J^{-})\otimes(C(\langle d,e\rangle),J^{++})\\
\text{using \ref{cor-cliff-inv} (a)}&\simeq&(C(\langle -dec,d,e\rangle),J^{+++})
\end{eqnarray*}
Thus it suffices to set $q=\langle -dec,d,e\rangle$. 
This implies the equivalence of (ii) and (iv).
In order to prove the equivalence of (ii) and (iii), note that according to the
above relations, we may take $(Q_0,\rho):=(C(\langle
b,-abc\rangle),J^{-+})$ or $(Q_0,\rho):=(C(\langle
d,e\rangle),J^{++})$ which are both the quaternion algebras with
orthogonal involutions. 
\end{proof}

\begin{prop}\label{tensor-second-kind}
Let $(A,\sigma)$ be a central simple algebra with unitary involution
over a field $K$. 
Let $k$ be the fixed field of $\sigma|_K$.
The following assertions are equivalent: 
\begin{itemize}
\item[(i)] There exist quaternion algebras with unitary involution
  $(A_1,\sigma_1),\cdots,(A_n,\sigma_n)$ over $K$ such that for all $i$ and
  $j$ we have $\sigma_i|_K=\sigma_j|_K$ and 
$$(A,\sigma)\simeq(A_1,\sigma_1)\otimes_K\cdots\otimes_K(A_n,\sigma_n).$$ 
\item[(ii)] There exist quaternion algebras with canonical involution  
  $(Q_1,\gamma_1),\cdots,(Q_n,\gamma_n)$ over $k$ such that
  \begin{equation}\label{equation-unitary}(A,\sigma)\simeq(Q_1,\gamma_1)\otimes_k\cdots\otimes_k(Q_n,\gamma_n)\otimes_k(K,\sigma|_K).\end{equation}
\item[(iii)] There exists a quadratic space $(V,q)$ of dimension
  $2n+1$ over $k$ with nontrivial discriminant such that 
$$(A,\sigma)\simeq(C(V,q),J_q^{\id}).$$ 
\end{itemize}
If $n$ is even, the assertions (i)-(iii) are equivalent to the following:  
\begin{itemize}
\item[(iv)] There exists a quadratic space $(V',q')$ of dimension
  $2n+2$ over $k$ such that 
$$(A,\sigma)\simeq(C_0(V',q'),J_{q'}^{\id}).$$
\end{itemize}
Moreover, if $n$ is odd the assertions (i)-(iii) are equivalent to the following:  
\begin{itemize}
\item[(v)] There exists a quadratic space $(V'',q'')$ of dimension
  $2n+2$ over $k$ and an orthogonal symmetry $\tau:V''\rightarrow V''$ such that 
$$(A,\sigma)\simeq(C_0(V'',q''),J_{q''}^{\tau}).$$
\end{itemize}
\end{prop}

\begin{proof}
The equivalence of (i) and (ii) is obvious. 
According to \ref{Albert-generalization} (ii), in
(\ref{equation-unitary}),
one may replace $(Q_1,\gamma_1)$ by $(Q'_1,\rho_1)$, where $(Q'_1,\rho_1)$ is
a suitable quaternion algebra with orthogonal involution over $k$.
This implies that one can write 
$$(A,\sigma)\simeq(A_0,\sigma_0)\otimes_k(K,\sigma|_K),$$ 
where $(A_0,\sigma_0)$ is a tensor product of $n$ quaternion algebras with
involution over $k$ and $\sigma_0$ can be chosen to be orthogonal or
symplectic.
Using \ref{tensor-odd} and \ref{tensor-even},  
there exists a quadratic space $(V,q)$
of dimension $2n$ over $k$ with $q\simeq\langle a_1,\cdots,a_{2n}\rangle$ such that 
\begin{equation}\label{A0sigma0}
(A_0,\sigma_0)\simeq
(C(V,q),J_q^{-\id}).
\end{equation}
There also exists an element $c\in k^{\times}\backslash k^{\times 2}$
such that $(K,\sigma)\simeq(C(\langle c\rangle),J^{-\id})$. 
Take $q'=\langle a_1,\cdots,a_{2n-1}\rangle$.
We thus obtain
\begin{eqnarray*}
(A_0,\sigma_0)\otimes(K,\sigma|_K)&\simeq&
(C(V,q),J_q^{-\id})\otimes(C(\langle c\rangle),J^{-\id})\\
\text{using \ref{cor-cliff-inv} (a)}&\simeq&(C(q\perp\langle
\d q\cdot c\rangle),J^{-\id\oplus\id}) \\
\text{using \ref{iso-cliff-inv} (a)}&\simeq&(C(q'),J^{-\id})\otimes(C(\langle a_{2n},\d q\cdot c\rangle),J^{-+})\\
\text{using \ref{qua-cliff-sym}, $\exists d,\ e\in k^{\times}$}&\simeq&(C(q'),J^{-\id})\otimes(C(\langle d,e\rangle),J^{++})\\
\text{using \ref{cor-cliff-inv} (a)}&\simeq&(C(-de\cdot q'\perp\langle d,e\rangle),J^{\id\oplus\id})
\end{eqnarray*}
This implies the equivalence of (i) and (iii).
Using \ref{C0C-id} and the relation (\ref{A0sigma0}), we obtain
$$(A_0,\sigma_0)\simeq(C_0(\langle-1\rangle\perp q),J^{\id}).$$

Suppose that $n$ is even.
Take $q_1=\langle-1\rangle\perp q$. 
We thus obtain
\begin{eqnarray*}
(A_0,\sigma_0)\otimes(K,\sigma|_K)&\simeq&
(C_0(q_1),J^{\id})\otimes(C(\langle c\rangle),J^{-\id})\\
\text{using \ref{cor-cliff-paire-inv}
  (a)}&\simeq&(C_0(q_1\perp\langle
-\d q_1\cdot c\rangle),J^{\id\oplus\id}). 
\end{eqnarray*}
Thus it suffices to put $q'=q_1\perp\langle-\d q_1\cdot\rangle$.
This implies the equivalence of (i) and (iv).

Suppose that $n$ is odd.
We can write 
$$(A,\sigma)=(Q_1,\rho_1)\otimes_k\cdots\otimes_k(Q_{n},\rho_{n})\otimes_k(K,\sigma|_K),$$
where $(Q_i,\rho_i)$, $i=1,\cdots,n$, are quaternion algebras with
orthogonal involution. 
According to (iii), there exists a quadratic space $(V,q)$ of
dimension $2n-1$ over $k$ such that
$$(C(V,q),J_{q}^{\id})\simeq(Q_2,\rho_2)\otimes_k\cdots\otimes_k(Q_{n},\rho_{n})\otimes_k(K,\sigma|_K).$$

According to \ref{qua-cliff-sym}, there exists a quadratic space
$(V_1,q_1)$ of dimension $2$ over $k$ 
such that 
$(C(V_1,q_1),J_{q_1}^{\id})\simeq(Q_1,\rho_1)$.
According to \ref{C0Cid}, there exists a reflection $\tau_1$ of the underlying
vector space of $\langle-1\rangle\perp q_1$ such that
$(C(q_1),J^{\id})\simeq(C_0(\langle-1\rangle\perp q_1),J^{\tau_1})$.

 Take $q'=\langle-1\rangle\perp q_1$. 
We thus obtain
\begin{eqnarray*}
(A,\sigma)&\simeq&(C_0(q'),J^{\tau_1})\otimes_k(C(V,q),J_{q}^{\id})
\\
\text{using \ref{cor-cliff-paire-inv}
  (a)}&\simeq&(C_0(q'\perp -\d q'\cdot q),J^{\tau_1\oplus\id}).
\end{eqnarray*}
Therefore it suffices to put $q''=q'\perp -\d q'\cdot q$ and
$\tau=\tau_1\oplus\id$. 
\end{proof}

\section{Type of involutions induced by orthogonal symmetries} \label{section-type}

\begin{prop}\label{type-cliff-sym}
\rm{(Compare with \cite[7.4]{shapiro})}  Let $(V,q)$ be a nondegenerate quadratic space of even dimension $n$.
Let $\sigma$ is an orthogonal symmetry
of $V$. 
Then the involution
  $J_q^\sigma$ is orthogonal if and only if 
$tr(\sigma)=n-2s \equiv
  0\ \text{or}\ 2\mod 8$, where $s$ is 
the dimension of the subspace ${V^-}$
  of the anti-symmetric elements of $V$ (cf. \ref{symetrie-car}).
\end{prop} 

\begin{proof}
If $n=2$, the equivalence of the conditions is obvious (cf. 
\ref{remark-deg-2}). So let assume that $n\geqslant 4$. 
At least one of the subspaces ${V^+}$ or ${V^-}$ is of dimension 
greater or equal to $2$. 
As ${V^+}$ and ${V^-}$ are orthogonal, relative to the bilinear form
associated to $q$,
we conclude that $q|_{{V^+}}$ and $q|_{{V^-}}$ are nondegenerate. Consider a 
subspace $W$ of dimension $2$ of $V$ with $q|_W$ nondegenerate, such that 
either $W\subset {V^+}$ or $W\subset {V^-}$. 
The orthogonal subspace $W^\perp$ is stable under $\sigma$. 
We thus have a 
decomposition $\sigma=\sigma|_W\oplus\sigma_{W^\perp}$. 
In order to simplify the notation,
let us set $\sigma_W=\sigma|_W$, $\sigma_{W^\perp}=\sigma|_{W^\perp}$,
$q_W=q|_W$ and $q_{W^\perp}=q|_{W^\perp}$. 
According to \ref{iso-cliff-inv}, we 
have an isomorphism of algebras with involution:
$$(C(V,q),J_q^\sigma)\simeq (C(W,q_W)\otimes C(W^\perp,q_{W^\perp}),
J_{q_W}^{\sigma_W}\otimes J_{\d q_W{\cdot}q_{W^\perp}}^{-\sigma_{W^\perp}})$$
If $W\subset {V^+}$, then we have $\sigma_W=\id$ and
$J_{q_W}^{\sigma_W}$ is consequently
orthogonal (cf. \ref{remark-deg-2}). 
In order to prove that $J_q^\sigma$ is
orthogonal, it is necessary and sufficient to show that
$J'=J_{\d q_W{\cdot}q_{W^\perp}}^{-\sigma_{W^\perp}}$ is orthogonal. 
By induction, $J'$ is orthogonal if and only if
$n'-2s'\equiv 0\ \text{or } 2 \mod 8$, where $n'=\dim W^\perp$ and 
$s'$ is the dimension of the subspace of the anti-symmetric elements of 
$-\sigma_{W^\perp}$. 
As $n'=n-2$ and $s'=n-2-s$, we obtain $n-2s=-(n'-2s')+2$ and the proof is achieved.

If $W\subset {V^-}$, we have $\sigma_W=-\id$ and $J_{q_W}^\sigma$ is
symplectic. 
In order to prove that $J_q^\sigma$ is orthogonal, it is necessary and
sufficient to show that $J'$ is symplectic. 
By induction, $J'$ is
symplectic if
$n'-2s'\equiv 4\text{ or }6\mod 8$.
 In this case, 
we have $n'=n-2$ and $s'=n-s$. 
Therefore we have 
$n-2s=-(n'-2s')-2$ and the proof is achieved.
\end{proof}

\begin{cor}\label{lewis-type}
(\rm{\cite[Prop. 3]{lewis-clifford}}) 
Let $(V,q)$ be a nondegenerate quadratic space 
of even
dimension $n$. 
Consider the involutions $J_q^{\id}$ and $J_q^{-\id}$ of $C(V,q)$.
Then we have:
\begin{itemize}
\item[\rm{(i)}] If $n\equiv 0 \mod 8$ then $J_q^{\id}$
  and $J_q^{-\id}$ are of orthogonal type.
\item[\rm{(ii)}] If $n\equiv 2 \mod 8$ then $J_q^{\id}$ is of
orthogonal type and $J_q^{-\id}$ is of symplectic type.
\item[\rm{(iii)}]  If $n\equiv 4 \mod 8$ then $J_q^{\id}$  and $J_q^{-\id}$
are of symplectic type.
\item[\rm{(iv)}] If $n\equiv 6 \mod 8$ then $J_q^{\id}$ is of
symplectic type and  $J_q^{-\id}$ is if orthogonal type.
\end{itemize} 
\end{cor}

\begin{cor}\rm{(\cite[8.4]{boi})}\label{type-inv-ind-eate}
Let $(V,q)$ be a nondegenerate quadratic space of
odd dimension $n$.
Consider the involution $\jz{q}{\id}$ of $C_0(V,q)$.
Then we have:
\begin{itemize}
\item[\rm{(i)}] If $n\equiv 1,7 \mod 8$ then $\jz{q}{\id}$ is
orthogonal.
\item[\rm{(ii)}] If $n\equiv 3,5 \mod 8$, then $\jz{q}{\id}$ is
symplectic.
\end{itemize} 
\end{cor}

\begin{proof}
Use \ref{lewis-type} and \ref{cor-cliff-iso}.
\end{proof}

\begin{cor}\label{type-reflection}
Let $(V,q)$ be a nondegenerate quadratic space of even
dimension $n$ and let $\tau$ be a reflection 
of $V$. 
Consider the involution $J_q^{\tau}$ of $C(V,q)$.
We have:
\begin{itemize}
\item[\rm{(i)}] If $n\equiv 2,4 \mod 8$, then $J_q^{\tau}$ 
is of orthogonal type.
\item[\rm{(ii)}] If $n\equiv 0,6 \mod 8$, then $J_q^{\tau}$ is of 
symplectic type.
\end{itemize}
\end{cor}

\begin{proof}
It suffices to note that for a reflection $\tau$, the dimension
$s$, of the subspace of the anti-symmetric elements is equal to $1$
and to use \ref{type-cliff-sym}. 
\end{proof}


\noindent


\scriptsize 
\noindent 
{\sc Department of Mathematical Sciences, 
Sharif University of Technology, P. O. Box: 11155-9415, Tehran, Iran. 
Homepage: {\tt http://sharif.ir/$\sim$mmahmoudi}.
Email address: {\tt mmahmoudi@sharif.ir}}

\end{document}